\documentclass[a4paper]{amsart}
\usepackage[utf8]{inputenc}
\usepackage[english]{babel}
\usepackage{amscd,amssymb}
\usepackage{amsmath}
\usepackage{amsthm}
\usepackage{graphicx}
\usepackage{fancybox}
\usepackage{amstext}
\usepackage{mathtools}
\usepackage{esint}
\usepackage[square,numbers]{natbib}
\usepackage{booktabs}
\usepackage[hidelinks]{hyperref}
\usepackage{comment}
\usepackage{tikz,pgfplots}
\usepackage{pgfplotstable}
\usepackage{subcaption}

\usepackage[noabbrev, capitalise, nameinlink]{cleveref}
\crefname{equation}{}{}
\crefname{assumption}{Assumption}{Assumptions}
\crefformat{equation}{\textup{#2(#1)#3}}

\newcounter{proofstep}
\crefname{proofstep}{Step}{Steps}
\Crefname{proofstep}{Step}{Steps}
\newcounter{proofblock}

\newcommand*{\proofstep}[1]{\refstepcounter{proofstep}\emph{Step \theproofstep\ (#1).}}
\newcommand*{\resetproofsteps}{\setcounter{proofstep}{0}\stepcounter{proofblock}}
\usepackage{multirow}
\usepackage{standalone}

\usepackage{dsfont}

\usepackage{pgf}
\pgfplotsset{compat=1.18}
\usepackage{lmodern}
\usepackage[dvipsnames]{xcolor}

\everymath=\expandafter{\the\everymath\displaystyle}
\IfFileExists{scrextend.sty}{
  \usepackage[fontsize=10.000000pt]{scrextend}
}{
  \renewcommand{\normalsize}{\fontsize{10.000000}{12.000000}\selectfont}
  \normalsize
}

\ifdefined\pdftexversion\else  \usepackage{fontspec}
\fi
\makeatletter\@ifpackageloaded{underscore}{}{\usepackage[strings]{underscore}}\makeatother

\newcommand*{\plotdatadir}{}

\usepackage{aliascnt}
\newtheorem{theorem}{Theorem}[section]

\newaliascnt{corollary}{theorem}

\aliascntresetthe{corollary}

\newaliascnt{lemma}{theorem}
\newtheorem{lemma}[lemma]{Lemma}
\aliascntresetthe{lemma}

\newaliascnt{example}{theorem}

\aliascntresetthe{example}

\theoremstyle{definition}

\newaliascnt{definition}{theorem}

\aliascntresetthe{definition}

\newaliascnt{assumption}{theorem}

\aliascntresetthe{assumption}

\theoremstyle{remark}

\newaliascnt{remark}{theorem}
\newtheorem{remark}[remark]{Remark}
\aliascntresetthe{remark}

\numberwithin{theorem}{section}
\numberwithin{equation}{section}
\numberwithin{figure}{section}
\numberwithin{table}{section}

\def\diam{\operatorname{diam}}

\def\with{\;\mid\;}

\DeclareMathOperator*{\esssup}{ess\,sup}

\newcommand*{\bvec}[1]{\bar{\vec{#1}}}
\newcommand{\norm}[1]{{\left\lVert{#1}\right\rVert}}
\newcommand{\jump}[1]{{[\![ #1 ]\!]}}

\newcommand{\setNodeDir}{\ensuremath{\setNode_\textup D}}
\newcommand{\setNodeD}{\setNodeDir}

\newcommand{\uD}{\ensuremath{\vu_\mathrm{D}}}
\newcommand{\lambdaD}{\ensuremath{\vec\lambda_\mathrm{D}}}

\newcommand{\rD}{\ensuremath{\vr_\mathrm{D}}}

\newcommand{\vLambda}{\ensuremath{\vec\Lambda}}

\renewcommand*{\vec}{\boldsymbol}

\newcommand*{\graph}{\ensuremath{\mathcal G}}
      \newcommand*{\Mgraph}{\ensuremath{\mathcal M_{\graph}}} \newcommand*{\Lgraph}{\ensuremath{\mathcal L_{\graph}}}     \newcommand*{\setEdge}{\ensuremath{\mathcal E}}
\newcommand*{\Edges}{\setEdge}
\newcommand*{\setNode}{\ensuremath{\mathcal N}}
\newcommand*{\edge}{\ensuremath{\mathfrak e}}
\newcommand*{\node}{\ensuremath{\mathfrak n}}

\newcommand{\delx}{\partial_{x}}
\newcommand{\R}{{\mathbb{R}}}

\newcommand{\Nodes}{\setNode}

\newcommand*{\vu}{\ensuremath{\vec u}}
\newcommand*{\vr}{\ensuremath{\vec r}}
\newcommand*{\vn}{\ensuremath{\vec n}}
\newcommand*{\vm}{\ensuremath{\vec m}}
\newcommand*{\vf}{\ensuremath{\vec f}}

\newcommand*{\del}{\ensuremath{\partial}}

\newcommand*{\wqtest}{{ \vec w_{\normalfont\textrm{q}}}}

\newcommand*{\wytest}{{\vec w_{\normalfont\textrm{y}}}}
\newcommand*{\wqtestd}{{ \bar{\vec w}_{\normalfont\textrm{q}}}}
\newcommand*{\wztestd}{{ \bar{\vec w}_{\normalfont\textrm{z}}}}
\newcommand*{\wytestd}{{ \bar{\vec w}_{\normalfont\textrm{y}}}}

\newcommand*{\vqd}{\ensuremath{\bar {\vec q}}}
\newcommand*{\vzd}{\ensuremath{\bar {\vec z}}}
\newcommand*{\vyd}{\ensuremath{\bar {\vec y}}}
\newcommand*{\vlD}{\ensuremath{{\vec \lambda}_\setNode}}
\newcommand*{\vlh}{\ensuremath{{\vec \lambda}}}
\newcommand*{\vel}{\ensuremath{\vec\eta_{\vec\lambda}}} \newcommand*{\vldD}{\ensuremath{\bar {\vec \lambda}_\setNode}}
\newcommand*{\vldh}{\ensuremath{\bar {\vec \lambda}}}

\newcommand*{\vue}{\ensuremath{\vu_\edge}}
\newcommand*{\vre}{\ensuremath{\vr_\edge}}

\newcommand*{\vPiq}{\vec{\Pi}^\mathrm{q}}
\newcommand*{\vPiy}{\vec{\Pi}^\mathrm{y}}
\newcommand*{\vq}{\vec q}
\newcommand*{\vy}{\vec y}
\newcommand*{\vw}{\vec w}

\newcommand*{\vix}{\vec{i}_{\times}}

\newcommand*{\Cy}{\vec{C}_\mathrm{y}}

\numberwithin{equation}{section}

\newenvironment{widedisplay}[1][60pt]{\begingroup
  \everydisplay\expandafter{\the\everydisplay
    \displaywidth=\dimexpr\displaywidth+#1\relax}}{\endgroup\ignorespacesafterend}

\AtBeginDocument{\def\MR#1{}
}

\newcommand*{\Hpw}{\ensuremath{\mathbf{H}^1_\mathrm{pw}(\Sigma)}}

\newcommand*{\Hone}{\ensuremath{\mathbf{H}^1(\Sigma)}}
\newcommand*{\Ltwoe}{\ensuremath{\mathbf{L}^2(\edge)}}
\newcommand*{\Ltwo}{\ensuremath{\mathbf{L}^2(\Sigma)}}
\newcommand{\polyp}{p}
\newcommand{\poly}{\mathbb{P}_\polyp}
\newcommand{\spoly}{\vec{V}^p}
\newcommand*{\vtheta}{\vec\theta}
\newcommand*{\veps}{\vec\epsilon}
\newcommand*{\rmq}{\mathrm{q}}
\newcommand*{\rmy}{\mathrm{y}}
\newcommand*{\vtq}{\vtheta^\rmq}
\newcommand*{\vty}{\vtheta^\rmy}
\newcommand*{\veq}{\veps^\rmq}
\newcommand*{\vey}{\veps^\rmy}
\newcommand*{\tDti}{\tfrac{2}{\Delta t}}
\newcommand*{\elly}[1]{\ell_{\mathrm{y}}^{#1}}
\newcommand*{\ellz}[1]{\ell_{\mathrm{z}}^{#1}}
\newcommand*{\Emax}{\ensuremath{{\mathfrak E}^\mathrm{max}}}

\title[An HDG method for
wave propagation in elastic beam networks]{A hybridizable discontinuous Galerkin method for
wave propagation in elastic beam networks}
\author[M.~Hauck, J.~Holten, A.~M{\aa}lqvist, A.~Rupp, L.~Swoboda]{Moritz Hauck$^*$, Joseph Holten$^*$, Axel M{\aa}lqvist$^\dagger$, Andreas Rupp$^\ddagger$, Lucia~Swoboda$^\dagger$}
\address{${}^*$ Institute for Applied and Numerical Mathematics, Karlsruhe Institute of Technology, Englerstr.~2, 76131 Karlsruhe, Germany}
\email{moritz.hauck@kit.edu}
\email{joseph.holten@kit.edu}
\address{${}^{\dagger}$ Department of Mathematical Sciences, Chalmers University of Technology \& University of Gothenburg, Chalmers Tvärgata 3, 412 96 Göteborg, Sweden}
\email{axel@chalmers.se}
\email{lucias@chalmers.se}
\address{${}^{\ddagger}$ Department of Mathematics, Faculty of Mathematics and Computer Science, Saarland University, Campus E1.1, 66123 Saarbrücken, Germany}
\email{andreas.rupp@uni-saarland.de}
\thanks{Corresponding author: Joseph Holten, \texttt{joseph.holten@kit.edu}}

\begin{document}

\begin{abstract}
This paper studies the numerical solution of elastic wave propagation on networks, modeled by elastodynamic equations posed on each edge, coupled at the nodes through suitable transmission conditions.
We propose and analyze a hybridizable discontinuous Galerkin method that exploits the network structure to reduce the global problem at each time step to a linear system whose size depends only on the number of network nodes and not on the polynomial degree of the discretization.
Combining it with an energy-conservative implicit time discretization, we derive a priori error estimates of optimal order in space and time.
The implicit time discretization avoids the severe CFL restriction caused by the large variation in fiber segment lengths.
To efficiently solve the resulting, typically ill-conditioned global system, we introduce a two-level overlapping additive Schwarz preconditioner.
Under suitable assumptions on the network, we establish uniform convergence of the resulting preconditioned conjugate gradient method. Numerical experiments confirm the theoretical findings.
\end{abstract}

\keywords{spatial network model, Timoshenko beam network, elastic wave propagation, hybridizable discontinuous Galerkin method, two-level domain decomposition}

\subjclass{35R02, 65M12, 65M15, 65M55, 65M60} 

\maketitle

\section{Introduction}
Many applications in science and engineering involve geometrically complex structures composed of slender, effectively one-dimensional components. Prominent examples include blood vessels~\cite{Blanco2014}, porous media~\cite{CEPT12}, and fiber-based materials~\cite{KMM20}. For such systems, resolving all microscopic details in a fully three-dimensional simulation is often computationally prohibitive.
A natural alternative is to represent the geometry as a spatial network, described by a graph~$\mathcal{G} = (\mathcal{N}, \mathcal{E})$ consisting of nodes and edges embedded in a bounded domain~$\Omega \subset \mathbb{R}^3$. The resulting spatial network model comprises one-dimensional differential equations posed along the edges, which are coupled via algebraic conditions at the nodes. In fiber-based materials, nodes are typically placed at fiber intersections, while edges represent the fiber segments connecting them; see~\cite{Picu2011,Kulachenko2012} for an overview of such models. To accurately capture the complex mechanics of fiber interactions, the graph topology may be refined by introducing additional nodes and edges at intersections. In the context of paper-based materials, this strategy has been employed in~\cite{KMM20,Grtz2022,Grtz2024} to accurately capture the stiffness of fiber bonds.
While spatial network models can be naturally coupled with bulk PDEs to describe, for example, fiber-reinforced composites~\cite{Khristenko2021} and vascular networks interacting with surrounding tissue~\cite{FKOWW22}, this work focuses on pure network models for fiber-based materials such as paper and cardboard.

In this paper, we consider elastic wave propagation in fiber networks as a model problem. The corresponding spatial network model comprises dynamic Timoshenko beam equations governing the elastic vibrations of individual beams~\cite{Timoshenko1921}, coupled at the network nodes by continuity conditions on displacements and rotations, as well as balance conditions for forces and moments~\cite{Lagnese1994}.
To discretize the spatial network model, we employ a hybridizable discontinuous Galerkin (HDG) method combined with an energy-conservative implicit time stepping scheme. For an overview of HDG methods, we refer to~\cite{CockburnGL09}; for applications to single Timoshenko beams and static networks of such beams, see~\cite{CelikerCS10} and~\cite{hauck2025}, respectively.
The primary advantage of such a hybrid approach in the network setting stems from the zero-dimensional nature of the coupling points: at each time step, the global system of equations can be reduced to a system defined only on the network nodes, independently of the polynomial degree of the discretization. The additional nodes introduced by the spatial discretization of the fiber segments can also be condensed out locally. A similar condensation strategy for diffusion-type PDEs on networks of hypersurfaces is detailed in~\cite{RuppGK22}.
The resulting condensed system of equations enforces balance conditions for the numerical fluxes at the network nodes. Here, independent local solvers on the edges map nodal unknowns to numerical fluxes, and the condensed system determines the unknowns that ensure these fluxes satisfy the balance conditions.
The resulting semi-discrete system is then advanced in time using an energy-conservative implicit time integration scheme, such as the Crank--Nicolson method. Although this choice is non-standard for wave-type problems, it is motivated by the large disparity in fiber segment lengths: with length ratios spanning several orders of magnitude, classical explicit schemes become impractical because of the severe Courant--Friedrichs--Lewy (CFL) time step restriction.

As an alternative to HDG, beam finite elements can be employed for the discretization of Timoshenko beam networks; see~\cite{Kapur1966,Davis1972,Thomas1975,Lees1982} for representative examples from the extensive literature. If too few degrees of freedom are used, these methods can suffer from numerical issues such as shear locking; see, for instance, the theoretical study in~\cite{Mukherjee2001}. Furthermore, for the stationary problem, in the special case of constant material coefficients and homogeneous or constant loads, analytical ansatz functions can be constructed from exact solutions of the beam equations on each edge; see, for example, \cite{Reddy1997,Jeleni2009,KufnerLSSS18}.

In this paper, we present an a priori error analysis of the proposed HDG method for the elastic beam network wave propagation problem, establishing optimal-order convergence rates for both spatial and temporal discretization errors under mesh and time step refinement. To this end, we employ a projection-based error analysis for HDG methods, which goes back to~\cite{Cockburn2010}. The analysis combines techniques developed for HDG discretizations of wave equations~\cite{Cockburn2014} with arguments tailored to the Timoshenko beam network setting.
Due to the complex geometry of the spatial network and potentially highly varying material coefficients, the linear systems arising from the HDG method at each time step are typically ill-conditioned. Numerical experiments indicate that standard black-box preconditioners, such as algebraic multigrid methods (see, e.g., the review~\cite{XZ17}), may fail to significantly accelerate convergence. This is likely because such approaches do not adequately capture the geometric structure of the underlying spatial network.
Limitations are also encountered with existing HDG-specific preconditioners~\cite{CockburnDGT2013,FabienKMR19,Lu2021,LuRK22a,Lu2023,WidleyMB21}, which rely on coarsening strategies that would fundamentally alter the topology of the network graph, thereby rendering the associated injection operators inapplicable.
To overcome these challenges, we employ a preconditioner based on the observation that the network behaves essentially as a continuous object at sufficiently coarse scales. This perspective enables one to bridge across scales by introducing an artificial coarse mesh over the network. Classical finite element preconditioning techniques, such as a two-level overlapping Schwarz preconditioner, can then be employed; see~\cite{GoHeMa22,hauck2025,görtz2025} in the context of spatial network models. The convergence analysis of this two-level preconditioner relies on a spectral equivalence result between the global system matrix at each time step and suitably weighted graph operators. Under assumptions on the network homogeneity and connectivity on coarse scales that are reasonable in the present setting, quasi-interpolation bounds can be established. Moreover, the spectral equivalence result allows one to transfer classical analysis techniques for Schwarz methods from the continuous setting (cf.~\cite{ToW05}) to the network case, thereby proving uniform convergence of the corresponding preconditioned conjugate gradient method.

Other scale-bridging approaches have also been developed for spatial network models, in particular multiscale methods such as the (Super-)Localized Orthogonal Decomposition; see~\cite{LODOG,MaP14,HeP13,HaPe21b,pumslod} in the continuous setting and~\cite{EGHKM24,HMM23,HaM22} in the context of spatial network models. For these methods, quasi-interpolation operators or related quantities of interest are explicitly employed in the construction. Relying on a compression property, they allow for a significant reduction in the size of the resulting linear systems, thereby making them tractable again, for instance by classical sparse direct~solvers.

The rest of this paper is organized as follows: In \cref{sec:modelproblem}, we introduce a Timoshenko beam network model describing elastic wave propagation. \Cref{sec:hybrid_dual_mixed} presents a hybrid dual mixed formulation, and \cref{sec:semidisc} describes the spatial discretization using an HDG method, including an a priori error analysis for the semi-discrete problem. In \cref{sec:fullydisc}, we carry out a time discretization and conduct the corresponding fully discrete error analysis. An efficient preconditioner is presented in \cref{sec:precond}, and numerical experiments are reported in \cref{sec:numexp}. The paper concludes with a summary and outlook in \cref{sec:conclusion}.

\section{Timoshenko beam network model} \label{sec:modelproblem}
In this section, we introduce the model problem considered in this work: a Timoshenko beam network model describing elastic wave propagation in fiber-based materials. Such network models provide a powerful tool for representing fiber-based materials, such as paper (see \cref{fig:paper} for an illustration). The underlying Timoshenko beam model, originally introduced in~\cite{Timoshenko1921}, extends the simpler Euler--Bernoulli beam model by incorporating shear deformation. This makes it applicable not only to slender beams but also to short and thick beams, where cross-sectional rotation and shear effects play a significant role.

\begin{figure}
	\centering
	\includegraphics[width=0.5\linewidth]{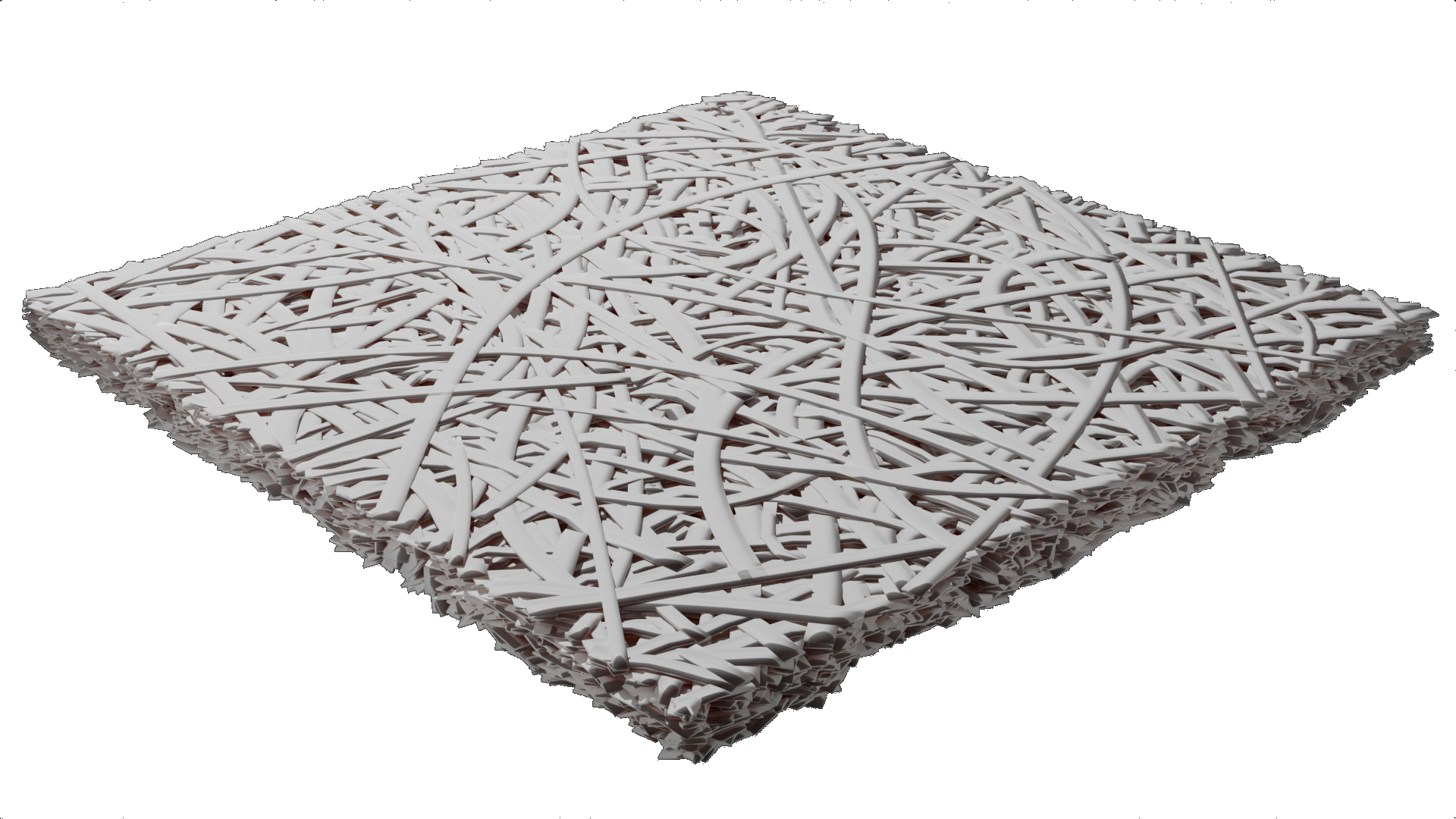}
	\caption{Illustration of a fiber network model of paper.}
	\label{fig:paper}
\end{figure}

\subsection{Geometry}
The underlying geometry is represented by a graph \(\graph=(\setNode,\setEdge)\) embedded in a bounded domain \(\Omega\subset\R^3\), where \(\setNode\) denotes the set of nodes and \(\setEdge\) denotes the set of initially straight, locally one-dimensional edges. The nodes represent fiber junctions, while the edges correspond to the fiber segments connecting pairs of nodes. We assume that the graph $\graph$ is connected, with the incidence set
\begin{equation*}
	\mathcal I \coloneqq \{ (\edge, \node) \in \setEdge \times \setNode \with \node \text{ is an endpoint of } \edge \}.
\end{equation*}
The network domain is denoted by
\begin{align*}
	\Sigma \coloneqq \bigcup_{\edge \in \setEdge} \bar{\edge},
\end{align*}
where $\bar{\cdot}$ denotes the closure of a set.
We describe the orientation of the beams by the edgewise constant function 
\(\vec{i}\colon\Sigma\to\mathbb{R}^3\), which assigns to each edge its unit tangent vector, oriented according to a fixed beam parametrization.
Subsequently, the canonical restrictions of a function $v$ and a space $V$ to an edge $\edge \in \setEdge$ are denoted by 
$v_\edge \coloneqq v|_\edge$ and $V_\edge \coloneqq V|_\edge$, respectively.

\subsection{Governing equations} \label{sec:gov_equations}
The governing equations of the network model are obtained by coupling the equations of motion on individual edges with suitable coupling conditions at the nodes. We consider the evolution of the system on the time interval $(0,T)$ with final time $T>0$.
On each edge, we employ the Timoshenko beam model~\cite{Timoshenko1921}, which is based on the kinematic assumptions that cross-sections remain rigid in their own plane and may rotate. In addition, the material is assumed to be linearly elastic, i.e., Hooke's law applies. 
The elastic and inertial properties of the network are characterized by the four coefficients 
\(\vec C_\mathrm{n}, \vec C_\mathrm{m}, \vec C_\mathrm{u}, \vec C_\mathrm{r} \colon \Sigma \to \mathbb{R}^{3\times 3}_\mathrm{sym}\),
which are assumed to be edgewise sufficiently smooth, uniformly bounded, and uniformly positive definite. More precisely,
 there exist constants $0 < c_{\bullet,\mathrm{min}} \leq c_{\bullet,\mathrm{max}} < \infty$ such that, for all $\vec \xi \in \mathbb{R}^3$ and almost every $\vec x \in \Sigma$,
\begin{align}\label{eq:boundCnm}
    c_{\bullet,\mathrm{min}}|\vec \xi|^2 &\leq \vec\xi^\top \! \vec C_\bullet(\vec{x}) \vec \xi \leq c_{\bullet,\mathrm{max}}|\vec \xi|^2,
\end{align}
for each $\bullet \in \{\mathrm{n,m,u,r}\}$, where $|\cdot|$ denotes the Euclidean norm.

The unknowns $\vu, \vr, \vn, \vm \colon (0,T) \times \Sigma \to \mathbb{R}^3$ of the beam network model
represent the displacement, rotation, force resultants, and moment resultants, respectively. On each edge $\edge\subset\Sigma$ and for almost every $t \in (0,T)$, they satisfy the equations
\begin{subequations}\label{eq:elasticwave}
   \begin{alignat}{2}
   - \vec C_\mathrm{n} ( \delx \vu + \vec{i} \times \vr) &= \vn,
   \qquad &
   - \vec C_\mathrm{m} \delx \vr &= \vm, \\
   \vec C_\mathrm{u} \del_{tt}\vu + \delx \vn &= \vf_\mathrm{n},
   \qquad &
   \vec C_\mathrm{r} \del_{tt}\vr + \delx \vm + \vec{i} \times \vn &= \vf_\mathrm{m};
\end{alignat}
see, e.g.,~\cite{Dym2013}. 
Here, $\vf_\mathrm{n}, \vf_\mathrm{m} \colon (0,T) \times \Sigma \to \mathbb{R}^3$ are given distributed forces and moments,
$\times$ denotes the standard cross product in $\mathbb{R}^3$,
$\delx$ the derivative with respect to the
parameter $x$ varying along the edge $\edge\subset\Sigma$ with unit speed, 
and $\del_t$ the partial derivative in time.
We impose the following initial conditions:
\begin{align}
\label{ic:weakwave}
    \vu(0,\cdot) &= \vu_0, & \vr(0,\cdot) &= \vr_0, &
    \del_t \vu(0,\cdot) &= \vec v_0, & \del_t \vr(0,\cdot) &= \vec s_0,
\end{align}
for given functions $\vu_0, \vr_0, \vec v_0, \vec s_0 \colon \Sigma \to \R^3$.

To couple the Timoshenko beam equations on the individual edges, we partition the set of nodes into a non-empty set of Dirichlet nodes $\setNodeDir\subset\setNode$ and the remaining nodes. At Dirichlet nodes, prescribed boundary conditions are imposed, while at the remaining nodes, continuity and balance conditions are enforced.
The \emph{continuity conditions} require the displacements and rotations at each node to be identical on all incident edges. Specifically, for each node $\node \in \setNode \setminus \setNodeDir$ and incident edges $\edge,\edge'\subset\Sigma$, we impose, for almost all $t \in (0,T)$, the conditions
\begin{align}\label{EQ:timo_cont_hybrid}
 \vue (t,\node) = \vu_{\edge^\prime}(t,\node), \qquad 
 \vre (t, \node) = \vr_{\edge^\prime}(t, \node).
\end{align}

The \emph{balance conditions} enforce equilibrium of forces and moments at non-Dirichlet nodes. To state them, we introduce the sign function $\nu \colon \mathcal I \to \{+1,-1\}$ defined by
\[
\nu(\edge,\node)=
\begin{cases}
	+1, & \text{if the tangent vector $\vec{i}_\edge$ points towards $\node$},\\
	-1, & \text{otherwise}.
\end{cases}
\]
We use the notation $\nu_\edge(\node)\coloneqq\nu(\edge,\node)$.
Moreover, for a $\nu$-weighted quantity defined on the edges, the jump operator $\jump{\cdot}(\node)$
denotes the sum of its traces at $\node$, that is,
\begin{align*}
    \jump{\vn \, \nu}(\node) = \sum_{\edge \colon (\edge,\node) \in \mathcal I} \vn_\edge(\node) \nu_\edge(\node).
\end{align*}
Then, for every $\node \in \setNode \setminus \setNodeDir$, the balance conditions read, for almost all $t\in(0,T)$,
\begin{align}\label{eq:balance}
   - \jump{ \vn \, \nu}(t,\node) = \vec 0, \qquad
  - \jump{\vm \, \nu}(t,\node) = \vec 0.
\end{align}
Together, the coupling conditions \cref{EQ:timo_cont_hybrid,eq:balance} enforce rigid connections between the fibers, which is why they are also referred to as rigid joint conditions.

The system is closed by \textit{Dirichlet boundary conditions}. Specifically, for each $\node\in\setNodeDir$ and almost every $t\in(0,T)$, we impose
\begin{equation}\label{eq:timo_dir}
	\vu(t,\node) = \uD(t,\node), \qquad \vr(t,\node) = \rD(t,\node),
\end{equation}
where $\uD,\rD\colon (0,T) \times \setNodeDir \to \mathbb{R}^3$ denote the prescribed Dirichlet data, which we identify with their zero extensions to $\setNode\setminus\setNodeDir$.
\end{subequations}

\section{Hybrid dual mixed formulation} \label{sec:hybrid_dual_mixed}
In this section, we introduce a hybrid dual mixed weak formulation that serves as the basis for the spatial HDG discretization developed in \cref{sec:semidisc}.
For notational convenience, pairs of functions in $\Sigma \to \mathbb{R}^3$
are identified with functions $\Sigma\to\mathbb{R}^6$.
Specifically, we introduce the following vector notation:
\begin{align*}
	\vec q = (\vn,\vm),\qquad 
	\vec y = (\vu,\vr),\qquad \vf \coloneqq (\vf_\mathrm{n},\vf_\mathrm{m}),
\end{align*}
where $\vec q$, $\vec y$, and $\vf$ encode the force and moment resultants, the displacement and rotation fields, and the external loads, respectively. Similar notation is applied for the boundary and initial data, denoting $\lambdaD \coloneqq (\uD,\rD)$, $\vec y_0 \coloneqq (\vu_0,\vr_0)$, and $\vec y_1 \coloneqq (\vec v_0, \vec s_0)$. Furthermore, we introduce the block matrices 
\begin{align}\label{eq:itimes}
	\vec C_\mathrm{q} \coloneqq \operatorname{diag}(\vec C_\mathrm{n},\vec C_\mathrm{m}),
	\qquad
	\vec C_\mathrm{y} \coloneqq \operatorname{diag}(\vec C_\mathrm{u},\vec C_\mathrm{r}).
\end{align}

\subsection{Function spaces and norms}
On each edge $\edge\subset\mathbb{R}^3$, we use the standard Lebesgue and Sobolev spaces
$\Ltwoe\coloneqq L^2(\edge;\mathbb{R}^6)$ and
$\mathbf{H}^p(\edge)\coloneqq H^p(\edge;\mathbb{R}^6)$, $p\in\mathbb{N}_{\geq1}$,
where all integrals are taken with respect to the one-dimensional Hausdorff measure.
Integrals over the network domain $\Sigma$ are understood as the sum of the corresponding edge integrals.
Accordingly, we define the global space $\Ltwo\coloneqq L^2(\Sigma;\mathbb{R}^6)$.
We denote by $(\cdot,\cdot)_\edge$ and $\|\cdot\|_\edge$
the $\Ltwoe$ inner product and induced norm, respectively, and use analogous notation on $\Sigma$.
The edgewise Sobolev space is defined by
\begin{align*}
	\mathbf{H}^p_\mathrm{pw}(\Sigma)
	&\coloneqq
	\bigl\{\vec\psi \in \Ltwo \with
	\vec\psi_\edge \in \mathbf{H}^p(\edge)
	\text{ for all } \edge\in\setEdge \bigr\}.
\end{align*}
A globally continuous version of this edgewise Sobolev space is defined by
\begin{align*}
	\mathbf{H}^p(\Sigma)
	&\coloneqq \mathbf{H}^p_\mathrm{pw}(\Sigma)
	\cap C^0(\Sigma;\mathbb{R}^6).
\end{align*}
For $p \geq 1$, we denote by $\|\cdot\|_{\edge,p}$ the $\mathbf{H}^p$ norm on $\edge$
and by $\|\cdot\|_{\Sigma,p}$ the corresponding piecewise $H^p$-norm on $\Sigma$.
Note that, since $\mathbf{H}^1(\edge)$ embeds continuously into
$C^0(\overline{\edge};\mathbb{R}^6)$ in one dimension, nodal traces are well defined.

We further define the nodal spaces
\begin{align*}
	\vLambda&\coloneqq\{\vec\lambda \;\colon\; \setNode\to\R^6
	\;\mid\; \vec\lambda=0 \text{ on } \setNodeDir\},\\
	\vec\Lambda_\mathrm{D}
	&\coloneqq
	\left\{    \vec\lambda \;\colon\; \setNode\to\R^6\;\mid\;
	\vec\lambda=0 \text{ on } \setNode\setminus\setNodeD
	\right\}.
\end{align*}
Here, $\vLambda$ consists of nodal functions satisfying homogeneous Dirichlet boundary conditions on $\setNodeDir$, while $\vec\Lambda_\mathrm{D}$ represents the space for Dirichlet data supported on~$\setNodeD$.

For vector-valued quantities $\vec\psi,\vec\phi \colon \Sigma \to \mathbb{R}^6$ with well-defined nodal traces, we introduce the following discrete inner products:
\begin{align*}
\langle \vec \psi, \vec \phi \rangle_\edge
&\coloneqq  \sum_{\node:\,(\edge,\node)\in\mathcal I} \vec\psi_\edge(\node)\cdot\vec\phi_\edge(\node),\\ \langle \vec \psi, \vec \phi \rangle_\setEdge &\coloneqq \sum_{\edge\in\setEdge}\langle\vec\psi,\vec\phi\rangle_\edge,\\
\langle\vec\psi,\vec\phi\rangle_\setNode  &\coloneqq \sum_{\node\in\setNode}\vec\psi(\node)\cdot\vec\phi(\node).
\end{align*}
The first two bilinear forms act on edgewise functions through their nodal traces, whereas the third one is defined for functions with unique nodal values. Note that in $\langle\cdot,\cdot\rangle_{\setEdge}$ each nodal contribution is counted with multiplicity equal to the degree of the corresponding node.
If $\vec \phi$ is single-valued at the nodes, the edge-based scalar product weighted by the orientation factor $\nu$ can be interpreted as the nodal inner product of the corresponding jump with the test function, i.e.,
\begin{align}
    \langle \vec \psi \, \nu, \vec\phi\rangle_\setEdge
    = \sum_{\edge\in\setEdge} \,\sum_{\node:\,(\edge,\node)\in\mathcal I} \vec\psi_\edge(\node) \, \nu_\edge(\node)\cdot\vec\phi_\edge(\node)
    = \langle \jump{\vec \psi \, \nu}, \vec\phi\rangle_\setNode. \label{eq:EtoN}
\end{align}

\subsection{Hybrid dual mixed formulation}
Given the initial data $\vec{y}_0, \vec{y}_1 \in \mathbf{H}^1(\Sigma)$ (cf.~\cref{ic:weakwave}), the force term $\vf \in H^1(0,T;\mathbf{L}^2(\Sigma))$, and the boundary data $\lambdaD \in H^2(0,T;\vec{\Lambda}_\mathrm{D})$ (cf.~\cref{eq:timo_dir}), we seek the dual, primal, and hybrid unknowns
\begin{align*}\vec q &\in L^2(0,T;\Hpw), \qquad
	\vec y \in H^2(0,T;\Ltwo), \qquad
	\vlh \in L^2(0,T;\vLambda),
\end{align*}
where $\vlD \coloneqq \vlh + \lambdaD$, satisfying the initial conditions
\begin{equation*}
	\vy(0) = \vy_0, \qquad \del_t \vy(0) = \vy_1,
\end{equation*}
such that, almost everywhere in $(0,T)$, the saddle-point system
\begin{subequations} \label{eq:saddle}
	\begin{alignat}{4}
		&\phantom{{}-{}} a(\vec q,\wqtest) &&+ b(\wqtest,\vec y)
		&&= -\langle \vec\lambda_\Nodes, \wqtest\nu \rangle_\setEdge,
		\label{eq:saddle-1}
		\\
		&- b(\vec q,\wytest) &&+ c(\del_{tt}\vec y,\wytest)
		&&= (\vf,\wytest)_\Sigma,
		\label{eq:saddle-2}
		\\
		& &&-\langle \vq \, \nu, \vec\mu\rangle_\setEdge &&= 0
		\label{eq:saddle-3}
	\end{alignat}
\end{subequations}
holds for all test functions $(\wqtest, \wytest, \vec\mu) \in \Hpw \times \Ltwo \times \vLambda$.
Here, \cref{eq:saddle-1} enforces the continuity conditions \eqref{EQ:timo_cont_hybrid} together with the boundary conditions \cref{eq:timo_dir}, whereas \cref{eq:saddle-3} enforces the balance conditions \eqref{eq:balance}; see also the identity \cref{eq:EtoN}.
The bilinear forms are defined as
\begin{alignat*}{5}
&a \colon \mathbf{H}^1_\mathrm{pw}(\Sigma)
&&\times \mathbf{H}^1_\mathrm{pw}(\Sigma)   &&\to \R,
& \qquad a(\vq,\vec{w}_\mathrm{q}) &\coloneqq (\vec C_\mathrm{q}^{-1} \vq, \vec w_\mathrm{q})_\Sigma
\\
& b \colon \mathbf{H}^1_\mathrm{pw}(\Sigma)
&&\times \mathbf{L}^2(\Sigma)               &&\to \R,
& b(\vq,\vy) &\coloneqq
    - (\del_x \vq + \vix\vq, \vy)_\Sigma
\\
&c \colon  \mathbf{L}^2(\Sigma)
&&\times \mathbf{L}^2(\Sigma)               &&\to \R,
& c(\vy,\vec{w}_\mathrm{y}) &\coloneqq
  (\vec{C}_\mathrm{y} \vy,\vec{w}_\mathrm{y})_\Sigma,
\end{alignat*}
where $\vix$ denotes the block matrix defined for all $\vec q = ( \vec n,\vec m)$ by $\vix\vec q \coloneqq (\vec 0,\, \vec i\times\vn)$, whose block transpose acts as $\vix^\top\vec y = (-\vec i\times\vr,\, \vec 0)$.
\begin{remark}[Spatial regularity of primal solution]\label{rem:add_reg}
Equation \eqref{eq:saddle-1} implies that, for almost every $t \in (0,T)$, the unknown $\vec y$ admits the spatial regularity $\vec y(t) \in \Hpw$ and that its nodal traces coincide with $\vlD(t)$. Since the latter is single-valued, in fact $\vec y(t) \in \Hone$.
\end{remark}

For the well-posedness of \cref{eq:saddle}, compatibility conditions must be imposed between the initial data
and the Dirichlet boundary data.
We define the initial dual data induced
by the constitutive relation \cref{eq:saddle-1} as 
\begin{align}
	\label{eq:comp_cond}
	\vec q_{m} &\coloneqq -\vec C_\mathrm{q}\bigl(\delx\vec y_{m} - \vix^\top\vec y_{m}\bigr)
    \qquad\text{for}~m\in\{0,1\}.
\end{align}
For $m\in\{0,1\}$, the \emph{compatibility conditions} require
\begin{subequations}\label{eq:comp_order}
\begin{alignat}{2}
	\del_t^m\lambdaD(0) &= \vec y_{m} & &\qquad\text{on}~\setNodeDir, \\
	\langle \vec q_{m}\,\nu,\vec\mu\rangle_\setEdge &= 0 & &\qquad\forall\vec\mu\in\vLambda,
\end{alignat}
\end{subequations}
where the last condition implicitly requires $\vec q_{m} \in \Hpw$.

The algebraic constraints linking the dual and hybrid variables preclude standard well-posedness results for the hybrid dual mixed formulation \cref{eq:saddle}. Consequently, we state the following well-posedness result.
\begin{lemma}[Continuous well-posedness]\label{lem:hdm_well}
	Assume that $\vf \in H^{1}(0,T;\Ltwo)$ and $\lambdaD \in H^{2}(0,T;\vec\Lambda_\mathrm{D})$, and that the compatibility conditions \cref{eq:comp_order} hold.
	Then, the hybrid dual mixed formulation \cref{eq:saddle} admits a unique solution satisfying
	\begin{align*}
		\vec q \in L^\infty(0,T;\Hpw),
		\qquad
		\vec y \in W^{2,\infty}(0,T;\Ltwo),
		\qquad
		\vlh \in W^{1,\infty}(0,T;\vLambda),
	\end{align*}
	with initial values $\vy(0)=\vy_0$ and $\del_t\vy(0) = \vy_{1}$.
\end{lemma}
\begin{proof}
For completeness, a proof sketch is provided in \cref{sec:wellposed}.
\end{proof}

\subsection{Energy conservation}
For sufficiently smooth solutions we define the energy
\begin{align}
\label{eq:energycont}
\mathfrak E[\vq, \vy](t) \coloneqq
\tfrac12 a(\vec q,\vec q)(t) + \tfrac12 c(\del_t\vec y,\del_t\vec y)(t).
\end{align}
Differentiating \cref{eq:saddle-1} with respect to time, testing the result with $\vec q$, testing \cref{eq:saddle-2} with $\del_t \vec y$, and summing the two equations yields
\begin{align*}
	\frac{d}{dt}\mathfrak E[\vq, \vy](t)
	= (\vec f,\del_t\vec y)_\Sigma
	- \langle \del_t\vec\lambda_\Nodes,
	\vec q\nu\rangle_{\setEdge}.
\end{align*}
In particular, for zero external forces ($\vf = \vec 0$) and time-independent Dirichlet data ($\del_t\lambdaD = \vec 0$), applying the balance conditions \cref{eq:saddle-3} yields global energy conservation:
\begin{align}
	\label{eq:energypreservationcont}
	\frac{d}{dt}\mathfrak E[\vq, \vy](t)=0,
\end{align}
as physically expected from the wave-like character of the underlying system.

\section{Semi-discrete problem}\label{sec:semidisc}

In this section, we derive a semi-discrete scheme by discretizing the hybrid dual mixed formulation \cref{eq:saddle} in space while keeping time continuous. We establish its well-posedness and energy conservation, and derive spatial error estimates.

\subsection{HDG discretization}
A natural choice of discretization in the present network setting is the HDG method.
Let $\poly[x]$ denote the space of univariate polynomials of degree at most $p$ in the arclength parameter $x$.
Given an edge $\edge \in \setEdge$ and any unit-speed parametrization $\gamma_\edge \colon [0, \operatorname{len}(\edge)]\to\edge$,
we define the local vector-valued polynomial space as
\begin{align*}
	\spoly_\edge \coloneqq \left\{ \vec w \colon \edge \to \mathbb{R}^6 \with w_i(\gamma_\edge(x)) \in \poly[x] \text{ for } i=1,\dots,6 \right\}.
\end{align*}
The global polynomial approximation space $\spoly \subset \{\Sigma \to \mathbb{R}^6\}$ is then given by the broken polynomial space over the network domain,
\begin{align*}
	\spoly \coloneqq \left\{ \vec w \colon \Sigma \to \mathbb{R}^6 \with \vec w|_\edge \in \spoly_\edge \text{ for all } \edge \in \setEdge \right\}.
\end{align*}
In addition, the discretization involves a stabilization parameter, which we take as a strictly positive, edgewise constant function $\tau\in L^\infty(\Sigma)$.
Here and in what follows, a bar accent $\bar{(\cdot)}$ denotes discrete quantities.

Given suitable discrete approximations $\bvec y_0, \bvec y_1 \in \spoly$ of the initial data $\vec y_0, \vec y_1$, we seek the discrete dual, primal, and hybrid unknowns
\begin{align*}\vqd \in L^2(0,T;\spoly), \qquad
	\vyd \in H^2(0,T;\spoly), \qquad
	\vldh \in L^2(0,T;\vLambda),
\end{align*}
where $\vldD \coloneqq \vldh+\lambdaD$.
The discrete unknowns are required to satisfy the initial conditions
\begin{equation*}
	\vyd(0) = \bvec y_0, \qquad \del_t \vyd(0) = \bvec y_1,
\end{equation*}
as well as the semi-discrete saddle-point system
\begin{subequations} \label{eq:sd_saddle}
	\begin{alignat}{4}
		&\phantom{{}-{}} a(\vqd,\wqtestd) &&+ b(\wqtestd, \vyd) && \label{eq:sd_saddle1}
		&&= -\langle \vldD,
		\wqtestd\nu \rangle_\setEdge, \\
		&-b(\vqd, \wytestd) &&+ c(\del_{tt}\vyd, \wytestd)
		&&+\langle\tau(\vyd-\vldD), \wytestd\rangle_\setEdge
		&&= (\vf, \wytestd)_\Sigma, \label{eq:sd_saddle2} \\
		& & & & & &\mathllap{-\langle \vqd\,\nu + \tau(\vyd-\vldD), \bvec\mu \rangle_\setEdge}
		&= 0 \label{eq:sd_saddle3}
	\end{alignat}
\end{subequations}
for all test functions $\wqtestd,\wytestd \in \spoly$ and $\bvec\mu\in\vLambda$, and almost every $t\in(0,T)$.
Here, \cref{eq:sd_saddle1} enforces the discrete continuity and boundary conditions, whereas \cref{eq:sd_saddle3} enforces the nodal balance conditions \eqref{eq:balance} in terms of the HDG numerical flux
\begin{align}\label{eq:numflux}
	\vqd \nu + \tau (\vyd - \bar{\vec \lambda}_\Nodes).
\end{align}
The stabilization parameter $\tau$ penalizes the jump between the primal trace $\vyd$ and the hybrid unknown $\bar{\vec \lambda}$. By the identity \cref{eq:EtoN}, \cref{eq:sd_saddle3} is equivalent to the balance conditions \cref{eq:balance} evaluated with the numerical flux \cref{eq:numflux}.

\begin{remark}[Role of stabilization]
Throughout, we assume that the stabilization parameter satisfies $\tau>0$.
Its strict positivity is used in the semi-discrete well-posedness analysis and provides the trace control required for the error analysis in \cref{sec:semidisc_error}.
The case $\tau=0$ is not considered.
\end{remark}

The semi-discrete counterpart of \cref{lem:hdm_well} reads as follows.
It provides one additional order of temporal regularity and requires no compatibility conditions on the data,
and its proof is considerably more direct:
after the spatial discretization, the algebraic constraints can be eliminated explicitly,
which reduces the problem to a system of ordinary differential equations.
\begin{lemma}[Semi-discrete well-posedness]\label{lem:sd_well}
Assume that $\vf \in H^{1}(0,T;\Ltwo)$ and that $\lambdaD \in H^{2}(0,T;\vec\Lambda_\mathrm{D})$.
Then the semi-discrete problem \cref{eq:sd_saddle}, together with the initial and boundary conditions, admits a unique solution, which satisfies
\begin{align*}
\vqd \in H^{2}(0,T;\spoly),
\qquad
\vyd \in H^{3}(0,T;\spoly),
\qquad
\vldh \in H^{2}(0,T;\vLambda).
\end{align*}
\end{lemma}

\begin{proof}
The bilinear form $a$ defines an inner product on $\Ltwo$ and, consequently, on the finite-dimensional subspace $\spoly$. 
By the Riesz representation theorem, equation \eqref{eq:sd_saddle1} uniquely determines the dual variable $\vqd = \bvec R(\vyd,\vldD)$ for given $\vyd$ and $\vldD$ through a linear solution operator, which we denote by $\bvec R$.
Inserting $\vqd$ into \eqref{eq:sd_saddle3} yields a linear system for $\vldh$ with the bilinear form
\begin{align*}
  B(\bvec\lambda,\bvec\mu) \coloneqq -\langle \bvec R(\vec 0,\bvec\lambda)\,\nu - \tau\bvec\lambda, \bvec\mu\rangle_\setEdge,
  \qquad \bvec\lambda,\bvec\mu\in\vLambda.
\end{align*}
This system is well-posed, as $B$ is symmetric positive definite.
Indeed, given $\bvec y = 0$ and $\bvec\lambda_\Nodes=\bvec\lambda$,
testing \eqref{eq:sd_saddle1} with $\bvec R(\vec 0,\bvec\mu)$
and exploiting the definition of $\bvec R$ yields 
\begin{align*}
  -\langle \bvec R(\vec 0,\bvec\lambda)\,\nu, \bvec\mu\rangle_\setEdge
  = a(\bvec R(\vec 0,\bvec\lambda),\bvec R(\vec 0,\bvec\mu)).
\end{align*}
For any $\bvec\lambda \neq \vec 0$, this identity implies the strict inequality
\begin{align*}
  B(\bvec\lambda,\bvec\lambda)
  = a(\bvec R(\vec 0,\bvec\lambda),\bvec R(\vec 0,\bvec\lambda))
  + \langle\tau\bvec\lambda,\bvec\lambda\rangle_\setEdge
  > 0,
\end{align*}
as every node has an incident edge and $\tau>0$.
Substituting $\vqd$ and $\vldh$ into \eqref{eq:sd_saddle2} yields a linear second-order system
of ordinary differential equations for $\vyd$ in $\spoly$
with symmetric positive definite mass form $c(\cdot,\cdot)$.
Since $\vf\in H^{1}(0,T;\Ltwo)$ and $\lambdaD\in H^{2}(0,T;\vec\Lambda_\mathrm{D})$,
the Picard--Lindelöf theorem yields a unique solution $\vyd\in H^{3}(0,T;\spoly)$
attaining the initial conditions, and the eliminated variables $\vqd$ and $\vldh$
inherit $H^{2}$-regularity in time.
Here, no compatibility conditions arise,
since the algebraic constraints \cref{eq:sd_saddle1} and \cref{eq:sd_saddle3}
determine the discrete dual and hybrid variables from any initial state
rather than restricting the admissible data.
\end{proof}

\subsection{Semi-discrete energy conservation}
We define the semi-discrete energy
\begin{align}
	\label{eq:energy}
	\mathfrak E_s[\vqd, \vyd, \bvec\lambda_\setNode](t) \coloneqq
	\tfrac12 a(\vqd,\vqd)(t) + \tfrac12 c(\del_t\vyd,\del_t\vyd)(t) + \tfrac12 \| \tau^{1/2} (\vyd - \vldD)\|^2_\setEdge.
\end{align}
Differentiating \cref{eq:sd_saddle1} with respect to time, testing the result with $\vqd$, testing \cref{eq:sd_saddle2} with $\del_t \vyd$, and summing the two equations yields
\begin{align*}
	\frac{d}{dt}\mathfrak E_s[\vqd, \vyd, \bvec\lambda_\setNode](t) = (\vec f,\del_t\bvec y)_\Sigma
	- \langle \del_t\vldD, \vqd\nu+\tau(\vyd-\vldD) \rangle_\setEdge.
\end{align*}
For zero external forces ($\vf = \vec 0$) and time-independent Dirichlet data ($\del_t\lambdaD = \vec 0$), applying the nodal balance conditions \cref{eq:sd_saddle3} yields global energy conservation:
\begin{align}
	\label{eq:energypreservation_sd}
	\frac{d}{dt}\mathfrak E_s[\vqd, \vyd, \bvec\lambda_\setNode](t) = 0,
\end{align}
analogous to the continuous setting \cref{eq:energypreservationcont}.

\subsection{Semi-discrete error analysis}\label{sec:semidisc_error}
The error analysis presented below relies on the projection-based HDG framework from~\cite{Cockburn2010}. To analyze the dynamic Timoshenko beam network problem, we combine this framework with techniques from the HDG error analysis for acoustic wave equations~\cite{Cockburn2014} and the static beam network setting~\cite{hauck2025}.
Central to this framework is the projection operator
\begin{align*}
  \vec\Pi_\edge = (\vec{\Pi}^\mathrm{q}_\edge ,\vec{\Pi}^\mathrm{y}_\edge)
  \colon \mathbf{H}^1(\edge) \times \mathbf{H}^1(\edge)
  \to \spoly_\edge \times\spoly_\edge,
\end{align*}
which preserves moments up to degree $p-1$ and mimics the discrete coupling structure at the endpoints through the stabilization parameter $\tau$. Specifically, for any $(\vq,\vy) \in \mathbf{H}^1(\edge) \times \mathbf{H}^1(\edge)$, the projection is defined as the unique pair satisfying
\begin{subequations}\label{eq:errproj}
\begin{align}
(\vPiq_\edge (\vq,\vy), \wqtestd)_\edge  &= (\vq,\wqtestd)_\edge,\\
(\vPiy_\edge (\vq,\vy), \wytestd)_\edge  &= (\vy,\wytestd)_\edge,\\
\vPiq_\edge (\vq,\vy)(\node)\nu_\edge(\node) + \tau_\edge \vPiy_\edge (\vq,\vy)(\node) &= \vq(\node)\nu_\edge(\node) + \tau_\edge \vy(\node)  \quad \text{ for } (\edge, \node) \in \mathcal I, \label{eq:node_condi}
\end{align}
\end{subequations}
for all test functions $\wqtestd,\wytestd \in \vec{V}^{p-1}_\edge$.
The following lemma, an adaptation of \citep[Thm.~2.1]{Cockburn2010}, establishes error estimates for this projection operator. Here and in what follows, $h_\edge \coloneqq \operatorname{len}(\edge)$ denotes the length of an edge $\edge \in \setEdge$, and $\tau_\edge \coloneqq \tau|_\edge > 0$ represents the local stabilization parameter on $\edge$.
\begin{lemma}[Projection error] \label{lem:proj_err}
  For any $p_\mathrm{q}, p_\mathrm{y}\in \{0, \dots, p\}$ and all
  $\vq \in \mathbf{H}^{p_{\mathrm{q}}+1}(\edge)$, $\vy \in \mathbf{H}^{p_{\mathrm{y}}+1}(\edge)$,
  the projection errors satisfy the bounds
  \begin{align*}
    \|\vq - \vPiq_\edge(\vq,\vy)\|_\edge
    &\lesssim \tau_\edge h^{p_\mathrm{y}+1}_\edge
    \|\vy\|_{\edge, {p_{\mathrm{y}}+1}}
    + h^{p_{\mathrm{q}}+1}_\edge \|\vq\|_{\edge, {p_{\mathrm{q}}+1}},\\
    \|\vy - \vPiy_\edge(\vq,\vy)\|_\edge
    &\lesssim h^{p_\mathrm{y}+1}_\edge \|\vy\|_{\edge, p_{\mathrm{y}}+1}
      + \tau^{-1}_\edge h^{p_{\mathrm{q}}+1}_\edge \|\vq\|_{\edge, p_{\mathrm{q}}+1}.
  \end{align*}
\end{lemma}

Globally, we split the error as
\begin{align*}
  \vq - \bar{\vq} &= \vtheta^\rmq + \veps^\rmq, &
  \vy - \bar{\vy} &= \vtheta^\rmy + \veps^\rmy,
\end{align*}
where the projection errors $\vtheta^\rmq, \vtheta^\rmy$
and the discrete errors $\veps^\rmq, \veps^\rmy$ are given by
\begin{align*}
  \vtq &\coloneqq \vec q - \vPiq(\vec q, \vec y), &
  \veq &\coloneqq \vPiq(\vec q, \vec y) - \vqd,\\
  \vty &\coloneqq \vec y - \vPiy(\vec q, \vec y),&
  \vey &\coloneqq \vPiy(\vec q, \vec y) - \vyd.
\end{align*}

We next derive the error equations for the semi-discrete problem that form the foundation of our subsequent error analysis.
\begin{lemma}[Error equations] \label{lem:semid_err_equ}
For all test functions $\wqtestd, \wytestd \in \spoly$ and almost all $t \in (0,T)$, the following identities hold:
\begin{align*}
    (\delx \wqtestd, \vey)_\Sigma
    &= a(\vec q-\vqd, \wqtestd)
    - (\vix \wqtestd,\vec y-\vyd)_\Sigma
    + \langle \vlD - \vldD, \wqtestd \, \nu \rangle_\setEdge,\\
    (\delx\veq, \wytestd)_\Sigma
    &= -(\vix(\vec q-\vqd), \wytestd)_\Sigma
    - c(\del_{tt} (\vec y - \vyd), \wytestd)
    + \langle\vy-\vldD-\vey, \tau\wytestd\rangle_\setEdge.
\end{align*}
\end{lemma}

\begin{proof}
The error equations follow directly by subtracting the continuous formulation~\cref{eq:saddle} from the semi-discrete problem~\cref{eq:sd_saddle}, integrating by parts, and applying the projection properties~\cref{eq:errproj}.
\end{proof}

The following theorem establishes the spatial convergence of the semi-discrete approximation with respect to the global mesh size parameter $h$, defined as
\begin{equation*}h \coloneqq \max_{\edge \in \setEdge} h_\edge.
\end{equation*}

\begin{theorem}[Semi-discrete convergence] \label{thm:semi_d_err}
Assume that the exact solution
$(\vec q, \vec y)$ satisfies $\vec q \in W^{2,\infty}(0,T; \mathbf{H}_\mathrm{pw}^{p+1}(\Sigma))$ and $\vec y \in W^{2,\infty}(0,T; \mathbf{H}^{p+1}(\Sigma))$, and choose $\bvec y_0,\bvec y_1\in\spoly$ as the $L^2$-projections
of $\vec y_0,\vec y_1$, respectively.
Let the stabilization parameter satisfy $\tau_\edge \simeq h_\edge^s$ for all $\edge \in \setEdge$ and some $s\in\{-1,0,1\}$.
Then, the semi-discrete HDG approximation $(\vqd,\vyd)$ converges to the exact solution with
\begin{align*}
  \esssup_{t \in [0,T]} \ &(\|\vec y(t)-\vyd(t)\|_\Sigma
  + \|\vec q(t)-\vqd(t)\|_\Sigma) \lesssim h^{p+1-|s|}.
\end{align*}
The hidden constant is independent of $h$ and depends on the coefficient bounds, quadratically on $T$, and on suitable norms of the exact solution.
\end{theorem}

\begin{proof}
By \cref{eq:EtoN}, the discrete balance condition \cref{eq:sd_saddle3} holds pointwise at every non-Dirichlet node. Taking its difference with \cref{eq:balance} and using the projection property \cref{eq:node_condi}, we obtain the condition
\begin{align}
    \jump{\veq \, \nu + \tau \vey - \tau(\vy - \vldD)}(t, \node) = 0, \label{eq:jump_diff}
\end{align}
for almost all $t \in (0,T)$ and all $\node \in \setNode \setminus \setNodeDir$.

In view of \cref{eq:EtoN}, this equation can be used when
testing the time derivative of the first error equation in \cref{lem:semid_err_equ} with $\veq$, as well as the second equation with $\del_t \vey$. Since the primal unknown $\vy$ belongs to $\Hone$ by \cref{rem:add_reg}, its nodal traces are single-valued and coincide with the hybrid unknown $\vlD$ at the network nodes. Comparing the expressions yields
\begin{align*}
        -\langle \del_t (\vy - \vldD), \veq \,\nu
        + \tau \vey
        - \tau (\vy - \vldD)\rangle_\setEdge
        = a(\del_t(\vq-\vqd), \veq)
        + c(\del_{tt}(\vy - \vyd), \del_t \vey) \quad\\
        - (\vix \veq, \del_t (\vy - \vyd))_\Sigma
        + (\vix (\vq - \vqd), \del_t \vey)_\Sigma
        + \langle \vey - (\vy - \vldD), \tau \del_t (\vey - (\vy - \vldD)) \rangle_\setEdge,
\end{align*}
where the term $\langle \del_t (\vy - \vldD), \tau \vey - \tau (\vy - \vldD)\rangle_\setEdge$ was added and subtracted so that \cref{eq:jump_diff} can be applied to the left-hand side. Furthermore,
the term $(\vix \veq, \del_t \vey)_\Sigma$ cancels. We obtain the identity
\begin{align*}
        a(\del_t \veq, \veq)
        + c(\del_{tt}\vey, \del_t \vey)
        + \langle \vey - (\vy - \vldD), \tau \del_t (\vey - (\vy - \vldD)) \rangle_\setEdge\\
        =
        - a(\del_t \vtq, \veq)
        - c(\del_{tt} \vty, \del_t \vey)
        + (\vix \veq, \del_t \vty)_\Sigma
        - (\vix \vtq, \del_t \vey)_\Sigma.
\end{align*}

This motivates the definition of the energy functional
\begin{align*}
    \mathfrak{E}_h(t)
    \coloneqq
    \tfrac12 a(\veq, \veq)(t)
    + \tfrac12 c(\del_t \vey,\del_t \vey)(t)
     + \tfrac12  \| \tau^{1/2} (\vey-(\vy-\vldD))(t)\|_\setEdge^2 .
\end{align*}
Taking the derivative of $\mathfrak E_h$ yields
\begin{align*}
    \frac{\partial \mathfrak E_h}{\partial t}(t)
    &=
    - a(\del_t \vtq, \veq)
    - c(\del_{tt} \vty, \del_t \vey)
    + (\vix \veq, \del_t \vty)_\Sigma
    - (\vix \vtq, \del_t \vey)_\Sigma\\
    &\lesssim
    (\| \del_t \vtq \|_\Sigma + \| \del_t \vty \|_\Sigma) \| \veq \|_\Sigma
    + (\| \del_{tt} \vty \|_\Sigma + \| \vtq \|_\Sigma) \| \del_t \vey \|_\Sigma.
\end{align*}
Using coercivity of $a$ and $c$ and the approximation bounds for the projection terms from \cref{lem:proj_err}, we obtain
    \begin{align*}
        \frac{\partial \mathfrak{E}_h}{\partial t}(t)
        \lesssim \sqrt{\mathfrak{E}_h(t)} \ h^{p+1-|s|}.
    \end{align*}
Therefore,
\begin{align*}
        \esssup_{t \in [0,T]}
        \big[
        \| \veq(t) \|^2_\Sigma
        + \| \del_t \vey(t)\|^2_\Sigma
        \big]^{1/2}
        \lesssim \esssup_{t \in [0,T]} \sqrt{\mathfrak{E}_h(t)}
        \lesssim T h^{p+1-|s|} + \sqrt{\mathfrak{E}_h(0)}.
\end{align*}
By the choice of the discrete initial data $\bvec y_0$ and $\bvec y_1$ as the $L^2$ projections of $\vec y_0$ and $\vec y_1$ onto $\spoly$, the initial dual and hybrid variables are recovered from \cref{eq:sd_saddle1} and \cref{eq:sd_saddle3} by solving the edgewise
defined local problems.
Consequently,
\begin{align*}
    \mathfrak{E}_h(0) \lesssim h^{2(p+1-|s|)}.
\end{align*}
The estimate
    \begin{align*}
        \| \vey(t)\|_\Sigma
        \leq \| \vey(0)\|_\Sigma
        + \int_0^t \| \del_{\tilde t} \vey(\tilde t) \|_\Sigma \,\mathrm{d}\tilde t
    \end{align*}
    completes the proof.
\end{proof}

\begin{remark}[Sharpness of the error estimate]\label{rem:sharpness}
The estimate in \cref{thm:semi_d_err} is optimal for the relevant choice $\tau_\edge\simeq 1$ ($s=0$), where it yields the rate $h^{p+1}$.
For $\tau_\edge\simeq h_\edge^{-1}$ ($s=-1$), the bound gives only the rate $h^p$. Numerical experiments in \cref{sec:numexp}, however, indicate one additional order of convergence in this case. This is consistent with the corresponding static Timoshenko beam problem considered in~\cite{hauck2025}, where the optimal rate of $h^{p+1}$ can be proved.
\end{remark}

\begin{remark}[Comparison with existing HDG formulations]
The first-order formulation considered here differs from existing formulations in the literature, such as those analyzed in~\cite{GriesmaierMonk2014}. There, two HDG formulations are proposed: a dissipative method based on the standard HDG numerical flux and an energy-conservative method based on a modified flux involving time derivatives. While the dissipative formulation admits optimal-order error estimates, the analysis of the energy-conservative formulation yields a convergence rate reduced by half an order; see~\citep[Thm.~6]{GriesmaierMonk2014}. This reduction is, however, not observed in their numerical experiments. In contrast, the present formulation is energy-conservative and, for the standard choice $\tau_\edge\simeq 1$, admits optimal-order error estimates.
We also note that additional superconvergence results could be obtained for HDG methods by local postprocessing together with a duality argument; see, e.g., \cite{Cockburn2014}. We do not pursue this here, since increasing the polynomial degree provides a straightforward way to improve the accuracy of all variables.
\end{remark}

\section{Fully discrete problem} \label{sec:fullydisc}
In this section, we derive a fully discrete scheme by discretizing the semi-discrete formulation in time. As a representative yet straightforward choice, we employ the Crank--Nicolson method. We emphasize that the following analysis extends naturally to the higher-order Gauss collocation methods, i.e., the collocation Runge--Kutta methods based on the Gauss--Legendre points, whose high-order convergence properties are demonstrated numerically in \cref{sec:numexp}.

Throughout this section, we additionally assume that $\vec C_\mathrm{u}$ and $\vec C_\mathrm{r}$, equivalently $\vec C_\mathrm{y}$, are edgewise constant. This assumption is made for simplicity, so that the second-order semi-discrete HDG formulation can be reformulated as an exactly equivalent first-order system in time using the auxiliary variable $\vzd = \Cy \del_t \vyd$.

\subsection{Time stepping}
Let $(t^k)_{k=0}^K$ be a uniform partition of $[0,T]$ with time step $\Delta t > 0$ and midpoint times $t^{k-1/2}=\tfrac12 (t^k+t^{k-1})$.
We introduce the bilinear form
\begin{align*}
d \colon \Ltwo\times\Ltwo \to \R, \qquad
d(\vy,\vw_\mathrm{y}) \coloneqq (\vec{C}_\mathrm{y}^{-1} \vy,\vec{w}_\mathrm{y})_\Sigma.
\end{align*}
The initial conditions for $\vyd$ and $\vzd$ are prescribed by
\begin{align*}\vyd(0) = \bvec y_0, \qquad \vzd(0) = \vzd_0,
\end{align*}
where $\vzd_0 \in \spoly$ satisfies
\begin{align*}
    d(\vzd_0, \wztestd) = (\bvec y_1, \wztestd)_\Sigma \qquad \forall \wztestd \in \spoly, \end{align*}
with \(\bvec y_0, \bvec y_1 \in \spoly\) denoting the discrete initial data from the second-order semi-discrete formulation.
This condition is the discrete counterpart of
\[\vzd(0) = \Cy \del_t \vyd(0) = \Cy \bvec y_1,\]
matching the initial auxiliary state of the first-order system.
\begin{remark}[Equivalence to second-order formulation]
The edgewise constancy of $\vec C_\mathrm{y}$ ensures that $\vec C_\mathrm{y}\spoly \subset \spoly$.
Hence, $\vzd \in \spoly$, and we obtain the exact equivalence of first- and second-order formulations.
For spatially varying $\vec C_\mathrm{y}$, the weak definition of $\vzd$ introduces a projection of $\Cy\del_t\vyd$ onto $\spoly$,
so that the exact equivalence is lost up to a consistency error, which would then have to be accounted for in the error analysis.
\end{remark}
In the first-order formulation, the semi-discrete saddle-point problem \cref{eq:sd_saddle} is replaced by seeking $(\vqd, \vyd, \vzd, \vldh)$, with $\vldD \coloneqq \vldh + \lambdaD$, such that
\begin{alignat*}{3}
    a(\vqd, \wqtestd) &+ b(\wqtestd, \vyd) && &&= -\langle \vldD, \wqtestd \nu \rangle_\setEdge, \\
    -b(\vqd, \wytestd) &+ \langle \tau (\vyd - \vldD), \wytestd \rangle_\setEdge &&+ (\del_t \vzd, \wytestd)_\Sigma &&= (\vf, \wytestd)_\Sigma, \\
    &- (\del_t \vyd, \wztestd)_\Sigma &&+ d(\vzd, \wztestd) &&= 0, \\
    &&&&\mathllap{-\langle \vqd \nu + \tau(\vyd - \vldD), \bvec\mu \rangle_\setEdge} &= 0
\end{alignat*}
 holds for almost every $t \in (0,T)$ and all test functions $\wqtestd, \wytestd, \wztestd \in \spoly$ and $\bvec\mu \in \vec\Lambda$.  
 
To formulate the Crank--Nicolson time stepping scheme, we introduce some notation: for a time-dependent function $\vec g$ and an integer $k$, we write $\vec g^k \coloneqq \vec g(t^k)$ whenever pointwise evaluation is well defined. Furthermore, for any integer-indexed family $(\vec y^k)_{k=0}^K$, in particular for the samples $(\vec g^k)_{k=0}^K$ of a time-dependent function, we define the midpoint average and discrete time derivative by
\begin{align*}
    \vy^{k-\frac12} \coloneqq \tfrac12 (\vy^k + \vy^{k-1}), \qquad
    D_t^{k-\frac12} \vy \coloneqq \frac{\vy^{k} - \vy^{k-1}}{\Delta t}.
\end{align*}

In addition to the initial values $\vyd^0$ and $\vzd^0$, the initial dual and hybrid variables $(\vqd^0,\vldh^0)$, with $\vldD^0 \coloneqq \lambdaD(0) + \vldh^0$, must also be computed. They are obtained from the initial constitutive relation and balance condition as follows:
\begin{align*}
  a(\vqd^0,\wqtestd) + b(\wqtestd,\vyd^0)
  &= -\langle \vldD^0,\wqtestd\nu\rangle_\setEdge
  \quad &&\forall \wqtestd\in\spoly,
  \\
  -\langle \vqd^0\nu + \tau(\vyd^0-\vldD^0), \bvec\mu\rangle_\setEdge
  &= 0
  \quad &&\forall \bvec\mu\in\vLambda.
\end{align*}
This system is uniquely solvable by the same static condensation and elimination arguments used in the proof of \cref{lem:sd_well}.

The Crank--Nicolson scheme yields the following discrete problem at time level~$t^k$: seek $\vqd^k, \vyd^k, \vzd^k \in \spoly$ and $\vldh^k \in \vLambda$, setting $\vldD^k \coloneqq \lambdaD(t^k) + \vldh^k$, satisfying
\begin{widedisplay}[36pt]
\begin{subequations}\label{eq:fd_midpoint}
\begin{alignat}{3}
  a(\vqd^k,\wqtestd) +{}& b(\wqtestd,\vyd^k) && &&= -\langle \vldD^k,\wqtestd\nu \rangle_\setEdge, \label{eq:fd1} \\
  -b(\vqd^{k-\frac12},\wytestd)
  +{}& \langle \tau (\vyd^{k-\frac12}-\vldD^{k-\frac12}),\wytestd \rangle_\setEdge &&+ (D_t^{k-\frac12} \vzd,\wytestd)_\Sigma &&= (\vf^{k-\frac12},\wytestd)_\Sigma, \label{eq:fd2} \\
  -&(D_t^{k-\frac12} \vyd, \wztestd)_\Sigma &&+ d(\vzd^{k-\frac12},\wztestd) &&= 0, \label{eq:fd3} \\
  -&\mathrlap{\langle \vqd^k\nu + \tau(\vyd^k-\vldD^k), \bvec\mu \rangle_\setEdge} && &&= 0 \label{eq:fd_balance}
 \end{alignat}
\end{subequations}
\end{widedisplay}
for all test functions $\wqtestd, \wytestd, \wztestd \in \spoly$ and $\bvec\mu \in \vLambda$.

Multiplying the second and third equations by two and moving all terms that depend on the previous time step to the right-hand side gives the equivalent formulation:\begin{subequations}\label{eq:saddle_fully_discrete}
\begin{alignat}{3}
  a(\vqd^k,\wqtestd) +{}& b(\wqtestd,\vyd^k) && &&= -\langle \vldD^k,\wqtestd\nu \rangle_\setEdge, \label{eq:saddle_fully_discrete-a}\\
  -b(\vqd^k,\wytestd)
  +{}& \langle \tau (\vyd^k-\vldD^k),\wytestd \rangle_\setEdge &&+ \tDti(\vzd^k,\wytestd)_\Sigma &&= (\vf^k,\wytestd)_\Sigma + \elly{k-1}(\wytestd), \label{eq:saddle_fully_discrete-b}\\
   -&\tDti(\vyd^k, \wztestd)_\Sigma &&+ d(\vzd^k,\wztestd) &&= \ellz{k-1}(\wztestd), \label{eq:saddle_fully_discrete-c}\\
   -&\mathrlap{\langle \vqd^k\nu + \tau(\vyd^k-\vldD^k), \bvec\mu \rangle_\setEdge} && &&= 0\label{eq:saddle_fully_discrete-d}
\end{alignat}
\end{subequations}
with the linear functionals $\elly{k-1}$ and $\ellz{k-1}$:
\begin{align*}
\elly{k-1}(\wytestd) &\coloneqq
(\vf^{k-1},\wytestd)_\Sigma
+ b(\vqd^{k-1},\wytestd) 
- \langle \tau(\vyd^{k-1}-\vldD^{k-1}),\wytestd\rangle_\setEdge
+ \tfrac{2}{\Delta t}(\vzd^{k-1},\wytestd)_\Sigma,
\\
\ellz{k-1}(\wztestd)
&\coloneqq
-d(\vzd^{k-1},\wztestd)
-\tfrac{2}{\Delta t}(\vyd^{k-1},\wztestd)_\Sigma.
\end{align*}

In contrast to the semi-discrete setting of \cref{lem:sd_well}, the well-posedness of the fully discrete system~\cref{eq:saddle_fully_discrete} at each time step follows from a purely algebraic argument, as established in the following lemma.
\begin{lemma}[Fully discrete well-posedness]\label{lem:fd_well}
  Let 
  $(\vqd^{k-1},\vyd^{k-1},\vzd^{k-1},\vldD^{k-1})$
  be given at the previous time level. Then, the fully discrete system \eqref{eq:saddle_fully_discrete}
  admits a unique solution $\vqd^k,\vyd^k,\vzd^k \in \spoly$, $\vldh^k\in\vLambda$.
\end{lemma}

\begin{proof}
The system is a square linear system on the finite-dimensional space $\spoly\times\spoly\times\spoly\times\vLambda$, so it suffices to show that the homogeneous problem admits only the trivial solution.
Testing the four equations with the respective unknowns $\vqd^k$, $\vyd^k$, $\vzd^k$ and $\vldh^k$ and adding them
leaves only the discrete energy \eqref{eq:disc_energy} of the solution,
while the right-hand side vanishes by the homogeneity of the data,
yielding
\begin{align*}
    a(\vqd^k,\vqd^k) + d(\vzd^k,\vzd^k) + \|\tau^{1/2}(\vyd^k-\vldh^k)\|_\setEdge^2 = 0.
\end{align*}
By the coercivity of $a$ and $d$, this yields $\vqd^k = \vec 0$ and $\vzd^k = \vec 0$, whence \cref{eq:saddle_fully_discrete-c} gives $\vyd^k=\vec 0$.
The remaining term then yields $\vldh^k = \vec 0$, since $\tau > 0$ and every node has at least one incident edge by the connectedness of the graph.
\end{proof}

\begin{remark}[Non-uniform time step]
This method extends naturally to non-uniform time steps, upon replacing $\Delta t$ by $\Delta t_k \coloneqq t^k-t^{k-1}$ in the definitions above.
\end{remark}

\subsection{Fully discrete energy conservation}\label{sec:energy_cons}

We define the fully discrete energy
\begin{align}\label{eq:disc_energy}
  \mathfrak{E}_d[\vqd, \vzd, \vyd, \vldD]
  \coloneqq \tfrac12 a(\vqd,\vqd) + \tfrac12 d(\vzd,\vzd)
  + \tfrac12 \|\tau^{1/2}(\vyd - \vldD)\|_\setEdge^2.
\end{align}

As shown in the following lemma, the Crank--Nicolson scheme directly inherits the energy-preserving structure of the semi-discrete formulation.

\begin{lemma}[Fully discrete energy identity]\label{lem:disc_energy}
Let $(\vqd^k, \vyd^k, \vzd^k, \vldD^k)_k$ solve \eqref{eq:fd_midpoint} with the right-hand sides of \eqref{eq:fd2} and \eqref{eq:fd3} replaced by $(\vec g_\mathrm{y}^{k-\frac12}, \wytestd)_\Sigma$ and $(\vec g_\mathrm{z}^{k-\frac12}, \wztestd)_\Sigma$, respectively (the scheme itself is recovered for $\vec g_\mathrm{y} = \vf$ and $\vec g_\mathrm{z} = \vec 0$). Then, with
\begin{align*}
    \mathfrak E_d^k\coloneqq\mathfrak E_d[\vqd^k,\vzd^k,\vyd^k,\vldD^k],
\end{align*}
we obtain the discrete energy identity
\begin{align} \label{eq:disc_energy_id}
\begin{split}
\mathfrak E_d^k-\mathfrak E_d^{k-1}
&= \Delta t\,(\vec g_\mathrm y^{k-\frac12},D_t^{k-\frac12}\vyd)_\Sigma
+ \Delta t\,(\vec g_\mathrm z^{k-\frac12},D_t^{k-\frac12}\vzd)_\Sigma\\
&\phantom{=}
-\Delta t \langle \vqd^{k-\frac12}\nu +\tau(\vyd^{k-\frac12}-\vldD^{k-\frac12}), D_t^{k-\frac12}\lambdaD \rangle_{\setEdge}.
\end{split}
\end{align}
\end{lemma}

\begin{proof}
We test the difference quotient of \eqref{eq:fd1} at the time levels $k$ and $k-1$ with $\wqtestd = \vqd^{k-\frac12}$, equation \eqref{eq:fd2} with $\wytestd=D_t^{k-\frac12}\vyd$, and \eqref{eq:fd3} with $\wztestd=D_t^{k-\frac12}\vzd$.
Adding the three identities cancels the terms involving $b$ and $(D_t^{k-\frac12}\vzd,D_t^{k-\frac12}\vyd)_\Sigma$. Subtracting $\langle \tau(\vyd^{k-\frac12}-\vldD^{k-\frac12}), D_t^{k-\frac12}\vldD\rangle_\setEdge$ on both sides yields
\begin{align*}
  a(D_t^{k-\frac12}\vqd,\vqd^{k-\frac12})
  &+ d(D_t^{k-\frac12}\vzd,\vzd^{k-\frac12})
  + \langle \tau(\vyd^{k-\frac12}-\vldD^{k-\frac12}), D_t^{k-\frac12}(\vyd-\vldD)\rangle_\setEdge\\
  &= (\vec g_\mathrm y^{k-\frac12},D_t^{k-\frac12}\vyd)_\Sigma
  + (\vec g_\mathrm z^{k-\frac12},D_t^{k-\frac12}\vzd)_\Sigma\\
  &\phantom{={}}
  - \langle \vqd^{k-\frac12}\nu + \tau(\vyd^{k-\frac12}-\vldD^{k-\frac12}), D_t^{k-\frac12}\vldD\rangle_\setEdge.
\end{align*}
Averaging \eqref{eq:fd_balance} at the time levels $k$ and $k-1$ and testing it with $\bvec\mu = D_t^{k-\frac12}\vldh\in\vLambda$ shows that only the Dirichlet part $D_t^{k-\frac12}\lambdaD$ of $D_t^{k-\frac12}\vldD = D_t^{k-\frac12}\lambdaD + D_t^{k-\frac12}\vldh$ contributes to the last term.
The identity
\[g(D_t^{k-\frac12}\bvec v,\bvec v^{k-\frac12})=\tfrac1{2\Delta t}\big(g(\bvec v^k,\bvec v^k)-g(\bvec v^{k-1},\bvec v^{k-1})\big)\]
holds for any symmetric bilinear form $g$. Applying it to $a$ with $\bvec v=\vqd$, to $d$ with $\bvec v=\vzd$, and to $\langle\tau\,\cdot,\cdot\rangle_\setEdge$ with $\bvec v = \vyd-\vldD$ then yields the desired discrete energy identity \eqref{eq:disc_energy_id}.
\end{proof}

As a consequence of \cref{lem:disc_energy}, the fully discrete Crank--Nicolson scheme conserves energy in the case of vanishing external forces (\(\vf=0\)) and time-independent Dirichlet data (\(D_t^{k-\frac12}\lambdaD = 0\)). In this case, we obtain
\begin{align*}
    \mathfrak E_d^k=\mathfrak E_d^{k-1}
    \qquad \forall  k\ge 1.
\end{align*}

This is the fully discrete counterpart of the energy conservation on the semi-discrete level, cf.~\eqref{eq:energypreservation_sd}.
It reflects the general fact that, for linear autonomous problems, the Crank--Nicolson scheme
coincides with the implicit midpoint rule, i.e., the one-stage Gauss collocation Runge--Kutta method,
which conserves all quadratic invariants exactly;
see, for example, \citep[Thm.~IV.2.1--2.2]{HairerLubichWanner06}.

\subsection{Static condensation}
Given the hybrid unknown $\vldD^k$, we introduce local solvers to reconstruct the interior state variables $(\vqd^k, \vyd^k, \vzd^k)$ inside each edge:
\begin{alignat*}{2} \vec Q^k&\colon (\vldD^k, \vf^k, \elly{k-1}, \ellz{k-1}) &&\mapsto \vqd^k, \\
\vec Y^k&\colon (\vldD^k, \vf^k, \elly{k-1}, \ellz{k-1}) &&\mapsto \vyd^k, \\
\vec Z^k&\colon (\vldD^k, \vf^k, \elly{k-1}, \ellz{k-1}) &&\mapsto \vzd^k.
\end{alignat*}
The local solvers are well defined and operate locally on each edge.
Indeed, for given data, the first three equations of \eqref{eq:saddle_fully_discrete} decouple into independent edgewise problems, because all involved bilinear forms are edgewise, and each of these problems is uniquely solvable by the argument in the proof of \cref{lem:fd_well}, applied on a single edge with $\vldD^k$ prescribed at both endpoints.
By linearity of the problem, splitting the Dirichlet data, external forcing, and data from previous time steps from the hybrid unknown yields
\begin{align*}
    (\vldD^k, \vf^k, \elly{k-1}, \ellz{k-1})
    &= (\vldh^k, \vec 0, 0, 0)
     + (\lambdaD^k, \vf^k, \elly{k-1}, \ellz{k-1}),\\
    \vec Q^k(\vldD^k, \vf^k, \elly{k-1}, \ellz{k-1})
    &= \vec Q^k(\vldh^k)
    + \vec Q_\mathrm{D}^k.
\end{align*}
Analogous notation is introduced for $\vec Y^k$ and $\vec Z^k$.

Substituting the local solvers into the equilibrium condition \eqref{eq:fd_balance} yields the fully discrete condensed formulation: at each time step $t^k$, find $\vldh^k\in \vLambda$ such that
\begin{align} \label{eq:global_f_d}
    \hat A(\vldh^k,\bvec\mu) = \hat F^k(\bvec\mu)
    \qquad \forall \bvec\mu\in \vLambda,
\end{align}
where
\begin{align}
  \hat A(\vldh^k,\bvec\mu)
    &\coloneqq -\langle
    {\vec Q}^k(\vldh^k)\,\nu + \tau({\vec Y}^k(\vldh^k)-\vldh^k)
  , \bvec\mu\rangle_\setEdge,\label{eq:Ahat}\\
  \hat F^k(\bvec\mu)
  &\coloneqq \langle {\vec Q}_\mathrm{D}^k\,\nu + \tau {\vec Y}_\mathrm{D}^k, \bvec\mu\rangle_\setEdge.
\end{align}

Although its definition \cref{eq:Ahat} is not manifestly symmetric, the bilinear form $\hat A$ can be rewritten in terms of the local solvers alone, which reveals it to be symmetric positive definite.
\begin{lemma}[Properties of the fully discrete condensed HDG problem]\label{lem:fd_glob}
The bilinear form $\hat A$ admits the equivalent representation
\begin{alignat}{1}
	\label{eq:refcondensed}
\begin{split}
\hat A(\vldh^k, \bvec \mu) 
&= 
a(\vec Q^k(\vldh^k),\vec Q^k(\bvec \mu))
 + d(\vec Z^k(\vldh^k), \vec Z^k(\bvec \mu))
\\
&\qquad 
+\langle \tau (\vec Y^k(\vldh^k)-\vldh^k),
\vec Y^k(\bvec \mu)-\bvec \mu
\rangle_\setEdge.
\end{split}
\end{alignat}
\end{lemma}

\begin{proof}
By definition, the triple $(\vec Q^k(\bvec\mu),\vec Y^k(\bvec\mu),\vec Z^k(\bvec\mu))$ solves \cref{eq:saddle_fully_discrete-a}--\cref{eq:saddle_fully_discrete-c}  with~$\vldD^k$ replaced by $\bvec\mu$ and vanishing source terms ($\vf^k = \vec 0,\, \elly{k-1} = 0,\, \ellz{k-1} = 0$). 
Testing the first equation (for datum $\bvec\mu$) with $\wqtestd = \vec Q^k(\vldh^k)$, the second (for datum~$\vldh^k$) with $\wytestd = \vec Y^k(\bvec\mu)$, and the third (for datum $\bvec\mu$) with $\wztestd = \vec Z^k(\vldh^k)$ yields
\begin{alignat*}{3}
    &\phantom{{}-{}} a(\vec Q^k(\bvec \mu),\vec Q^k(\vldh^k)) &&+ b(\vec Q^k(\vldh^k), \vec Y^k(\bvec \mu))
    &&= -\langle \bvec \mu, \vec Q^k(\vldh^k)\,\nu \rangle_\setEdge, \\
    &- b(\vec Q^k(\vldh^k),\vec Y^k(\bvec \mu))
    &&+ \tDti(\vec Z^k(\vldh^k),\vec Y^k(\bvec\mu))_\Sigma
    &&= -\langle\tau(\vec Y^k(\vldh^k)-\vldh^k), \vec Y^k(\bvec \mu)\rangle_\setEdge, \\
    &-\tDti(\vec Z^k(\vldh^k),\vec Y^k(\bvec\mu))_\Sigma &&+d(\vec Z^k(\vldh^k),\vec Z^k(\bvec\mu)) &&= 0.
\end{alignat*}
Adding the first two identities and inserting the third one, we rearrange
and subtract $\langle \tau (\vec Y^k(\vldh^k)-\vldh^k),\bvec \mu \rangle_\setEdge$ on both sides,
which yields \cref{eq:refcondensed}.
\end{proof}

\subsection{Fully discrete error analysis}\label{sec:error_analysis}
The following error analysis for the fully discrete scheme relies on the $L^\infty$-error estimate of the semi-discrete formulation and standard energy techniques for the Crank--Nicolson scheme.

\begin{theorem}[Fully discrete convergence]\label{thm:fully_d}
Assume that the exact solution satisfies
\begin{align*}
  \vq \in C^{4}([0,T];\mathbf{H}_\mathrm{pw}^{p+1}(\Sigma)), \quad
  \vy \in C^{4}([0,T];\mathbf{H}^{p+1}(\Sigma)).
\end{align*}
Let the stabilization parameter satisfy $\tau_\edge \simeq h_\edge^s$ for $s\in\{-1,0,1\}$ and all edges $\edge\in\Edges$.
Then,
the fully discrete HDG approximation $(\bar{\vy}^{k},\bar{\vq}^{k})$ at time $t^k$ converges to the exact solution with the error estimate
\begin{align*}
\max_{0\le k\le K}
(\|\vec q(\cdot,t^k)-\bar{\vec q}^{k}\|_\Sigma + \|\vy(\cdot,t^k)-\bar{\vy}^{k}\|_\Sigma)
&\lesssim h^{p+1-|s|} + (\Delta t)^2.
\end{align*}
The hidden constants depend on the coefficient bounds, the final time $T$ and the regularity of the exact solution, but are independent of $h$ and $\Delta t$.
\end{theorem}

\begin{remark}[Regularity requirements]\label{rem:temporal_reg}
The assumed temporal regularity can be proved using a standard semigroup bootstrap argument:
each additional time derivative of the solution requires one more time derivative of the data
together with a compatibility condition on the initial data,
cf.~\citep[\S 2.2.4, Prop.~4.51, Prop.~7.28]{Mugnolo14}.
Since $\del_t^2\vy$ solves the same equation, these conditions amount to higher piecewise spatial regularity and to the requirement that the acceleration field satisfies the same transmission conditions as the displacement.
\end{remark}
\begin{proof}
    The result is obtained by a standard error splitting into spatial and temporal error components, i.e.,
    \begin{align*}
        \|\vec q(\cdot,t^k)-\bar{\vec q}^{k}\|_\Sigma
        \leq
        \|\vec q(\cdot,t^k)-\bvec q(t^k)\|_\Sigma
        + \|\bvec q(t^k)-\bar{\vec q}^{k}\|_\Sigma,
    \end{align*}
and analogously for $\vy$. In both cases, the first term can be estimated by \cref{thm:semi_d_err}.
For the second term, we set $\vec e_q^k\coloneqq \bvec q(t^k)-\bar{\vec q}^{\,k}$
and analogously for \(\vec e_y^k\), \(\vec e_z^k\), and \(\vec e_\lambda^k\);
since both formulations use the same Dirichlet data, the Dirichlet parts cancel and
\(\vec e_\lambda^k\in\vLambda\).
\resetproofsteps
The proof proceeds in three steps: we first show that the discrete energy of the error controls all error components, then bound this energy by the temporal consistency error, and finally show that the latter is of the asserted order.
The temporal consistency error of a function \(\vec v\in C^1([0,T];\Ltwo)\) is the defect of the trapezoidal rule, viewed as a function of the midpoint,
\begin{align} \label{eq:cons_err}
\vec\rho[\vec v](t)
\coloneqq
\tfrac12\big(\del_t\vec v(t+\tfrac{\Delta t}2)+\del_t\vec v(t-\tfrac{\Delta t}2)\big)
- (\Delta t)^{-1}\big(\vec v(t+\tfrac{\Delta t}2)-\vec v(t-\tfrac{\Delta t}2)\big),
\end{align}
where we define $\vec\rho_y^{k-\frac12}\coloneqq\vec\rho[\bvec y](t^{k-\frac12})$,
and analogously $\vec\rho_z^{k-\frac12}$.
Evaluating the first and fourth semi-discrete equations at $t^k$, averaging the second and third equations at the time levels $t^k$ and $t^{k-1}$, and subtracting the fully discrete equations give the system
\begin{widedisplay}[36pt]
\begin{subequations} \label{eq:fd_err_sys}
\begin{alignat}{3}
  a(\vec e_q^k,\wqtestd) +{}& b(\wqtestd,\vec e_y^k) && &&
  = -\langle \vec e_\lambda^k,\wqtestd\,\nu \rangle_\setEdge,
  \\
  -b(\vec e_q^{k-\frac12},\wytestd)
  +{}& \langle \tau (\vec e_y-\vec e_\lambda)^{k-\frac12},\wytestd \rangle_\setEdge
  &&+ (D_t^{k-\frac12}\vec e_z,\wytestd)_\Sigma
  &&= -(\vec\rho_z^{k-\frac12},\wytestd)_\Sigma,
  \\
  -&(D_t^{k-\frac12}\vec e_y,\wztestd)_\Sigma
  &&+ d(\vec e_z^{k-\frac12},\wztestd)
  &&= (\vec\rho_y^{k-\frac12},\wztestd)_\Sigma, \label{eq:fd_err_sys3}
  \\
  -&\mathrlap{\langle \vec e_q^k\,\nu + \tau(\vec e_y^k-\vec e_\lambda^k), \bvec\mu \rangle_\setEdge} && && = 0.
\end{alignat}
\end{subequations}
\end{widedisplay}
The first and the last equations carry no consistency error, as they hold at every time level for both formulations.
Thus, the error system is the fully discrete scheme of \cref{lem:disc_energy} with vanishing Dirichlet data and right-hand sides $\vec g_\mathrm y=-\vec\rho_z$ and $\vec g_\mathrm z=\vec\rho_y$, and all errors vanish for $k=0$, since the initial values of the fully discrete scheme are defined by the same relations as those of the semi-discrete solution.
We denote the error energy by $\mathfrak E_d^k\coloneqq\mathfrak E_d[\vec e_q^k,\vec e_z^k,\vec e_y^k,\vec e_\lambda^k]$ as in \eqref{eq:disc_energy} and abbreviate
\[E\coloneqq\max_{0\le k\le K}\sqrt{\mathfrak E_d^k}.\]

\proofstep{The energy controls the error}\label{stp:fd_energy}
Since $\mathfrak E_d^k \ge \tfrac12 a(\vec e_q^k,\vec e_q^k) + \tfrac12 d(\vec e_z^k,\vec e_z^k)$, the coercivity of $a$ and $d$ gives $\|\vec e_q^k\|_\Sigma + \|\vec e_z^k\|_\Sigma \lesssim E$ for all $0\le k\le K$.
For $\vec e_y$, testing \eqref{eq:fd_err_sys3} with $\wztestd = D_t^{k-\frac12}\vec e_y\in\spoly$ and using the boundedness and coercivity of $d$ yields
\begin{align} \label{eq:fd_dtey}
    \|D_t^{k-\frac12}\vec e_y\|_\Sigma
    \lesssim
    \|\vec e_z^{k-\frac12}\|_\Sigma + \|\vec\rho_y^{k-\frac12}\|_\Sigma
    \lesssim
    E + \|\vec\rho_y^{k-\frac12}\|_\Sigma.
\end{align}
Since $\vec e_y^0=\vec 0$, the telescoping identity $\vec e_y^n = \Delta t\sum_{k=1}^n D_t^{k-\frac12}\vec e_y$ and $n\Delta t\le T$ give
\begin{align} \label{eq:fd_reduction}
    \max_{0\le k\le K}\big(\|\vec e_q^k\|_\Sigma + \|\vec e_z^k\|_\Sigma + \|\vec e_y^k\|_\Sigma\big)
    \lesssim
    E + \max_{1\le k\le K}\|\vec\rho_y^{k-\frac12}\|_\Sigma,
\end{align}
so that it suffices to bound the error energy and the temporal consistency error.

\proofstep{The consistency error bounds the energy}\label{stp:fd_bound}
The discrete energy identity \eqref{eq:disc_energy_id} and summation over $k=1,\dots,n$ give
\begin{align*}
    \mathfrak E_d^n = -\sum_{k=1}^n \Delta t\,(\vec\rho_z^{k-\frac12},D_t^{k-\frac12}\vec e_y)_\Sigma
    + \sum_{k=1}^n \Delta t\, (\vec\rho_y^{k-\frac12},D_t^{k-\frac12}\vec e_z)_\Sigma.
\end{align*}
In the second sum, $D_t^{k-\frac12}\vec e_z$ is not controlled by the energy, and we use summation by parts instead; since $\vec e_z^0=\vec 0$, this gives
\begin{align*}
    \sum_{k=1}^n \Delta t\, (\vec\rho_y^{k-\frac12},D_t^{k-\frac12}\vec e_z)_\Sigma
    =
    (\vec\rho_y^{n-\frac12}, \vec e_z^n)_\Sigma
    -\sum_{k=1}^{n-1} (\vec\rho_y^{k+\frac12}- \vec\rho_y^{k-\frac12}, \vec e_z^k)_\Sigma.
\end{align*}
Abbreviating
\begin{align*}
    R &\coloneqq \max_{1\le k\le K}\big(\|\vec\rho_y^{k-\frac12}\|_\Sigma + \|\vec\rho_z^{k-\frac12}\|_\Sigma\big)
    + \sum_{k=1}^{K-1}\|\vec\rho_y^{k+\frac12}-\vec\rho_y^{k-\frac12}\|_\Sigma,
\end{align*}
the Cauchy--Schwarz inequality, the estimate \eqref{eq:fd_dtey}, the bound $\|\vec e_z^k\|_\Sigma\lesssim E$ of \cref{stp:fd_energy} and $n\Delta t\le T$ yield $\mathfrak E_d^n \lesssim R\,E + R^2$ for all $n\le K$.
Taking the maximum over $n$ gives $E^2 \lesssim R\,E + R^2$ and therefore, by Young's inequality,
\begin{align} \label{eq:fd_energy_bound}
    E \lesssim R.
\end{align}

\proofstep{The consistency error is of the asserted order}\label{stp:fd_order}
For $\vec v\in C^3([0,T];\Ltwo)$ and $\vec w \in C^1([0,T];\Ltwo)$, the defect \eqref{eq:cons_err} satisfies
\begin{align*}
    \|\vec\rho[\vec v](t^{k-\frac12})\|_\Sigma
    &\lesssim (\Delta t)^2 \max_{t\in[t^{k-1},t^{k}]}\|\del_t^3\vec v(t)\|_\Sigma,
    \\ \|\vec\rho[\vec w](t^{k-\frac12})\|_\Sigma
    &\lesssim \max_{t\in[t^{k-1},t^{k}]}\|\del_t\vec w(t)\|_\Sigma,
\end{align*}
the first one by the error representation of the trapezoidal rule, the second one since $D_t^{k-\frac12}\vec w$ is an average of $\del_t\vec w$ over $[t^{k-1},t^k]$.
Since $\del_t\vec\rho[\vec v]=\vec\rho[\del_t\vec v]$,
the increments
\[\vec\rho[\vec v](t^{k+\frac12})-\vec\rho[\vec v](t^{k-\frac12})
=\int_{t^{k-\frac12}}^{t^{k+\frac12}}\vec\rho[\del_t\vec v](t)\,\mathrm{d}t\]
obey the same two bounds with one more derivative and an additional factor $\Delta t$.
Splitting the semi-discrete solution as $\bvec y = \vy + (\bvec y - \vy)$ and
$\bvec z = \vec z + (\bvec z - \vec z)$, where $\vec z\coloneqq \Cy\del_t\vy$, and
applying the first estimate to the exact solution and
the second one to the semi-discrete error, we obtain
\begin{align*}
    R &\lesssim (\Delta t)^2 \max_{t\in[0,T]}
    \big( \|\del_t^3\vy\|_\Sigma + \|\del_t^4 \vy\|_\Sigma \big)\\
    &\phantom{\lesssim{}}
    + \max_{t\in[0,T]}
    \big( \|\del_t(\bvec y-\vy)\|_\Sigma
    + \|\del_t^2(\bvec y-\vy)\|_\Sigma
    + \|\del_t(\bvec z-\vec z)\|_\Sigma \big),
\end{align*}
where we used $\del_t^3\vec z = \Cy \del_t^4\vy$ and absorbed the factor $K\Delta t = T$
stemming from the sum of the increments into the constant.
The first term is of order $(\Delta t)^2$ by the assumed regularity of the exact solution.
For the second one, note that, by the linearity of the problem and
the time independence of all bilinear forms,
$\del_t^j\bvec y$ is the semi-discrete approximation of $\del_t^j\vy$
for the data $\del_t^j\vf$ and $\del_t^j\lambdaD$;
the assumed regularity is the one required in \cref{thm:semi_d_err} for $j=1,2$.
The differentiated problems start with an initial error of the corresponding order,
since their initial data is obtained from $\bvec y_0$, $\bvec y_1$ by the semi-discrete equations,
which reproduce the corresponding relations of the continuous problem up to the projection error.
That theorem thus gives
\[\max_{t}\|\del_t^j(\bvec y-\vy)\|_\Sigma \lesssim h^{p+1-|s|}.\]
The bound for $\del_t(\bvec z-\vec z)$ follows from
$d(\bvec z,\wztestd) = (\del_t\bvec y,\wztestd)_\Sigma$,
which also holds for $\vec z$ and all $\wztestd\in\Ltwo$:
the coercivity of $d$ bounds the difference of $\del_t\bvec z$
and the $d$-orthogonal projection of $\del_t\vec z$ onto $\spoly$ by
$\|\del_t^{2}(\bvec y - \vy)\|_\Sigma$,
while the projection error is of order $h^{p+1}$.
Altogether,
\[R \lesssim h^{p+1-|s|} + (\Delta t)^2.\]

By the definition of $R$, inserting \eqref{eq:fd_energy_bound}
into \eqref{eq:fd_reduction} bounds all error components by $R$,
which is the asserted estimate.
\end{proof}

\section{Domain decomposition preconditioner}\label{sec:precond}

A practical implementation of the proposed HDG method requires solving the linear system of equations corresponding to~\cref{eq:global_f_d}. In the present context of simulating paper-based materials, cf.~\cref{fig:paper}, this can easily become intractable for direct solvers due to their large size. In addition, these systems are typically poorly conditioned, requiring appropriate preconditioners when using iterative methods. This section introduces a two-level overlapping additive Schwarz preconditioner similar to~\cite{GoHeMa22}, which we use within a preconditioned conjugate gradient method. 

The domain decomposition and coarse space used by the proposed preconditioner are constructed using an artificial (coarse) mesh $\mathcal{T}_H$ of a bounded domain $\Omega\subset \mathbb{R}^3$ containing the spatial network. 
\begin{figure}
 \includegraphics[height=.5\textwidth]{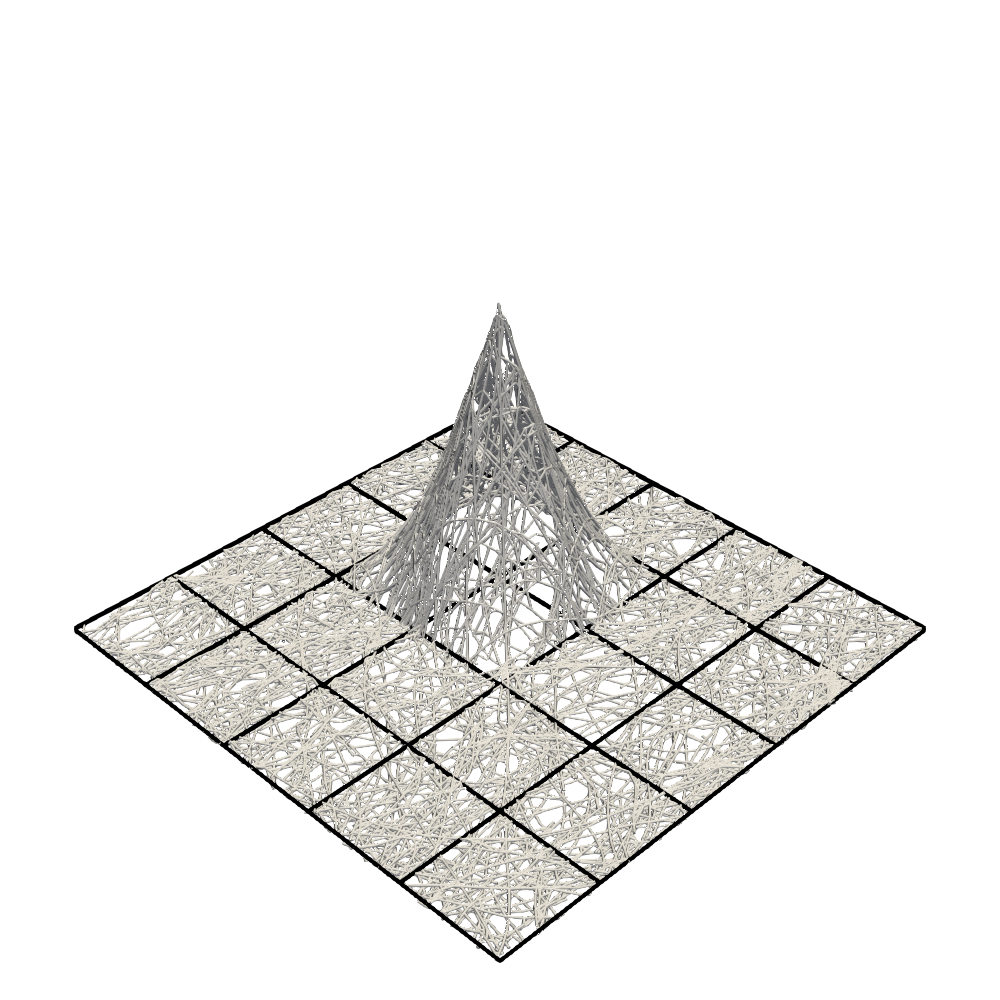}
 \caption{An artificial mesh $\mathcal{T}_H$ (black) overlaid on the network (gray),
 with a coarse basis function $\varphi_i$.}
 \label{fig:artificialgrid}
\end{figure}
For simplicity, we assume that $\Omega$ is a box equipped with a uniform Cartesian mesh.  
Note that for thin materials such as cardboard, the corresponding artificial mesh may contain far fewer elements in one spatial dimension than in the others.
With respect to the artificial mesh, we then introduce the set of trilinear basis functions $\{\varphi_i\}_{i=1}^m$, where $m$ is the number of nodes of the artificial mesh, and corresponding supports~$U_i\coloneqq \operatorname{supp}(\varphi_i)$. An illustration of an artificial mesh and corresponding basis functions can be found in \cref{fig:artificialgrid}. The space of continuous piecewise trilinear functions with respect to $\mathcal{T}_H$, satisfying Dirichlet boundary conditions on boundary segments where the network nodes are fixed, is denoted by $V_H$. Note that the domain of the functions in~$V_H$ is considered to be the nodes of the spatial network. 

We employ a preconditioner based on the subspace decomposition
\begin{equation*}
	\vLambda=\vLambda_{0}+ \vLambda_{1}+\dots+ \vLambda_{m},
\end{equation*}
with the coarse space $\vLambda_{0}\coloneqq V_H^6$ and local subspaces defined for $i=1,\dots,m$ by $\vLambda_{i} \coloneqq \{\vec v\in \vLambda \with \operatorname{supp}(\vec v)\subset U_i\}$. Given this decomposition, for each subspace we introduce a corresponding projection operator $P_i\colon \vLambda \rightarrow \vLambda_{i}$ such~that
\begin{equation*}
	\hat A  (P_i \vec \lambda,\vec \mu)=\hat A(\vec \lambda,\vec \mu), 
\end{equation*}
holds for all $\vec \mu \in \vLambda_{i}$, cf.~\cref{eq:Ahat}.
Note that the existence and uniqueness of such an operator are a direct consequence of $\hat A$ being an inner product on $\vLambda$. A preconditioned version of the operator $\hat A$ can then be defined as follows:
\begin{equation*}
	P\coloneqq P_0+P_1+\dots+P_m.
\end{equation*}
The preconditioner, denoted by $B$, is then given by the relation $P=B\hat A$. Note that the preconditioner is never explicitly formed. In practice, only the preconditioned operator must be computed, which requires the direct solution of a coarse global problem and~$m$ local problems, all of which can be solved independently. Because this approach involves direct solves, it is sometimes called semi-iterative. In the following, we use this preconditioner for the conjugate gradient method.

In this section, we prove that the operator $\hat A$ is spectrally equivalent
to a combination of graph mass-type and weighted graph Laplacian-type operators;
for the static setting, see~\citep[Ass.~3.1]{GoHeMa22}.
Certain homogeneity, connectivity, and locality assumptions on the spatial network at coarse scales
\citep[Ass.~3.5, Lem.~3.7]{GoHeMa22}
then yield uniform convergence of the resulting preconditioned conjugate gradient method~\citep[Thm.~4.3]{GoHeMa22}.
The proof in~\citep{GoHeMa22} is inspired by classical Schwarz theory (see, e.g.,~\cite{Xu92,ToW05,KorY16})
and constructs a quasi-interpolation operator in the spatial network setting,
whose approximation and stability properties can be proved using Friedrichs' and Poincar\'e's
inequalities on subgraphs~\citep[Lem.~3.6]{GoHeMa22}.
In practice, this preconditioner has demonstrated its ability to cope with the typically complex
geometry of spatial networks and highly varying material properties;
see, e.g.,~\cite{Grtz2024}.
The spectral equivalence announced above is formulated in terms of two bilinear forms,
which we introduce first.
Both are assembled from local contributions associated with an edge $\edge\in\setEdge$
with endpoints $\node_i,\node_j\in\Nodes$:
a mass-type form and a weighted graph Laplacian-type form.
For $\vec \lambda, \vec \mu \in \vLambda$, we define
\begin{align*}
    \mathcal M_{\graph,\edge}(\vec \lambda,\vec \mu)
    &\coloneqq \tfrac12 h_\edge (
        \vec \lambda(\node_i) \cdot \vec \mu(\node_i)
        + \vec \lambda(\node_j) \cdot \vec \mu(\node_j)
    ),\\
    \mathcal L_{\graph,\edge}(\vec \lambda,\vec \mu)
    &\coloneqq \tfrac12 h_\edge^{-1}
        (\vec \lambda(\node_i)-\vec \lambda(\node_j))
        \cdot
        (\vec \mu(\node_i)-\vec \mu(\node_j)).
\end{align*}

The corresponding global bilinear forms are obtained by summation over all edges, i.e.,
\begin{align}
    \Mgraph(\vec \lambda,\vec \mu)
    \coloneqq \sum_{\edge \in \setEdge}
    \mathcal M_{\graph,\edge}(\vec \lambda,\vec \mu),\qquad
    \Lgraph(\vec \lambda,\vec \mu)
    \coloneqq \sum_{\edge \in \setEdge}
    \mathcal L_{\graph,\edge}(\vec \lambda,\vec \mu). \label{eq:lapl_mass}
\end{align}

The following theorem establishes that the condensed problem \cref{eq:global_f_d}
is spectrally equivalent to a combination of the mass-type form $\Mgraph$ and the weighted graph Laplacian $\Lgraph$.
This is exactly what one would typically expect from discretizations of wave-type problems.
For classical finite element discretizations, for example, the graph operators would be replaced by corresponding mass and stiffness matrices.

\begin{theorem}[Spectral equivalence] \label{thm:spectral_equi}
  Assume that the polynomial degree satisfies $\polyp\geq1$,
  that the time step satisfies $h \lesssim \Delta t$,
  and that the material coefficients $\vec C_\bullet(\vec{x})$ are edgewise constant
  for $\bullet \in \{\mathrm{n,m,u,r}\}$.
Then, there exist $0<\alpha\leq\beta<\infty$ such that for all $\vec\lambda \in \vLambda$,
  we have
	\begin{align}
		\label{EQ:speceq}
		\alpha\left(\frac{\Mgraph(\vec \lambda,\vec \lambda)}{(\Delta t)^2}
    + \Lgraph(\vec \lambda,\vec \lambda)\right)
    \leq
    \hat A(\vec \lambda,\vec \lambda)
    \leq \beta\left( \frac{\Mgraph(\vec \lambda,\vec \lambda)}{(\Delta t)^{2}} + \Lgraph(\vec \lambda,\vec \lambda)
     \right),
	\end{align}
	where the constants $\alpha,\beta$ depend only on the material parameters, the diameter of the domain,
  the final time $T$, and the constant implicit in the assumed scaling of $\Delta t$.
\end{theorem}

\begin{proof}

We set
\begin{align*}
c_{\mathrm q,\min}&\coloneqq \min\{c_{\mathrm n,\min},c_{\mathrm m,\min}\}, & c_{\mathrm q,\max}&\coloneqq \max\{c_{\mathrm n,\max},c_{\mathrm m,\max}\},
\end{align*}
and analogously define $c_{\mathrm y,\min}$ and $c_{\mathrm y,\max}$.
For a fixed time step, we suppress the superscript $k$ and denote by $\vec Q(\vlh)$, $\vec Y(\vlh)$, and $\vec Z(\vlh)$
the edge-local solvers from
\cref{eq:saddle_fully_discrete-a,eq:saddle_fully_discrete-b,eq:saddle_fully_discrete-c}
with hybrid datum $\vlh$ and zero data.
To prove the desired bounds, we use a variational representation in terms of a local energy
of the edge-local condensed bilinear form
\begin{align*}
    \hat A_\edge(\vlh, \vlh)
    = a_\edge(\vec Q(\vlh), \vec Q(\vlh))
    + d_\edge(\vec Z(\vlh), \vec Z(\vlh))
    + \tau_\edge \langle  \vec Y(\vlh) - \vlh, \vec Y(\vlh) - \vlh \rangle_\edge.
\end{align*}
For any $\vec\eta\in\spoly_\edge$, let
$\vec Q_{\vec\eta},\vec Z_{\vec\eta}\in\spoly_\edge$
be determined by \eqref{eq:saddle_fully_discrete-a} and \eqref{eq:saddle_fully_discrete-c},
respectively, with primal variable $\vec\eta$ and hybrid datum $\vlh$; in particular
$\vec Q_{\vec Y(\vlh)}=\vec Q(\vlh)$ and $\vec Z_{\vec Y(\vlh)}=\vec Z(\vlh)$.
Explicitly, we have $\vec Z_{\vec\eta} = \tfrac{2}{\Delta t} \Cy\vec\eta$ and
integration by parts on the edge yields for all $\vec\eta,\wqtestd\in\spoly_\edge$
\begin{equation}
\begin{aligned}
  a_\edge(\vec Q_{\vec\eta},\wqtestd)
    &= (\delx\wqtestd+\vix\wqtestd, \vec\eta)_\edge
       -\langle \vlh, \wqtestd\nu \rangle_\edge \\
    &= -(\delx\vec\eta-\vix^\top\vec\eta, \wqtestd)_\edge
       +\langle \vec\eta-\vlh, \wqtestd\nu \rangle_\edge.
\end{aligned}
\label{eq:Qeta2}
\end{equation}
The local energy is defined as
\begin{align*}
\mathfrak E_\edge(\vec\eta)
\coloneqq
\tfrac{1}{2}a_\edge(\vec Q_{\vec\eta},\vec Q_{\vec\eta})
+\tfrac{1}{2}d_\edge(\vec Z_{\vec\eta},\vec Z_{\vec\eta})
+\tfrac{1}{2}\tau_\edge
\langle\vec\eta-\vlh,\vec\eta-\vlh\rangle_\edge.
\end{align*}
Differentiating and using \eqref{eq:saddle_fully_discrete-a} and
\eqref{eq:saddle_fully_discrete-c} gives, for every $\vec\xi\in\spoly_\edge$,
\begin{align*}
\mathfrak E_\edge'(\vec\eta)[\vec\xi]
= -b_\edge(\vec Q_{\vec\eta},\vec\xi)
+\tfrac{2}{\Delta t} (\vec Z_{\vec\eta},\vec\xi)_\edge
+\tau_\edge \langle\vec\eta-\vlh,\vec\xi\rangle_\edge.
\end{align*}
Hence, by \eqref{eq:saddle_fully_discrete-b},
$\vec Y(\vlh)$ is a stationary point of $\mathfrak E_\edge$.
By the coercivity of $d_\edge$ and \eqref{eq:saddle_fully_discrete-c},
$\mathfrak E_\edge$ is strictly convex, and hence $\vec Y(\vlh)$ is its unique minimizer.
Therefore
\begin{align}
\hat A_\edge(\vlh,\vlh)
= 2\mathfrak E_\edge\bigl(\vec Y(\vlh)\bigr)
\leq 2\mathfrak E_\edge(\vec\eta)
\qquad \forall\vec\eta\in\spoly_\edge. \label{eq:upper-min}
\end{align}

The upper bound follows directly by choosing as competitor
the componentwise affine interpolant of the endpoint values of $\vlh$,
denoted by $\vec\eta_{\vec\lambda}$.
This choice makes the stabilization term in the local energy vanish
and makes \cref{eq:Qeta2} yield
$\vec Q_{\vel} = -\vec C_\mathrm{q}(\delx\vel-\vix^\top\vel)$,
the discrete counterpart of the constitutive relation \cref{eq:comp_cond}.
This allows us to write the upper bound as
\begin{align*}
    \hat A_\edge(\vlh,\vlh)
    \leq 2\,\mathfrak E_\edge(\vel)
    = \norm{\vec C_\mathrm{q}^{1/2}(\delx\vel-\vix^\top\vel)}_\edge^{2}
      + \tfrac{4}{(\Delta t)^2} \norm{\Cy^{1/2}\vel}_\edge^{2}.
\end{align*}
The properties of the affine interpolant and the bounds on the material coefficients
yield the upper bound
\begin{align*}
    \hat A_\edge(\vlh,\vlh)
    \leq 4c_{\mathrm q,\max}\mathcal L_{\graph,\edge}(\vlh,\vlh)
    + \Bigl(2c_{\mathrm q,\max}+\frac{4c_{\mathrm y,\max}}{(\Delta t)^2}\Bigr)
    \mathcal M_{\graph,\edge}(\vlh,\vlh),
\end{align*}
where $\Delta t\leq T$ allows absorption with
$\beta \coloneqq \max\{4c_{\mathrm{q},\max}, 2c_{\mathrm{q},\max}T^2 + 4c_{\mathrm{y},\max}\}.$

For the lower bound, we disregard the non-negative stabilization term in the representation of $\hat A_\edge$,
yielding
\begin{align}\label{eq:lower_start}
    \hat A_\edge(\vlh,\vlh)
    \geq a_\edge(\vec Q(\vlh),\vec Q(\vlh))
       + \frac{4}{(\Delta t)^2} c_{\mathrm y,\min}\norm{\vec Y(\vlh)}_\edge^{2}.
\end{align}
To bound the first term from below, we use the dual representation of a norm in a Hilbert space:
taking any scalar polynomial $0\not\equiv\phi(x)\in\poly[x]$ and a constant $\vec 0\neq\vec c\in\R^6$,
we can write
\begin{align*}
  a_\edge(\vec Q(\vlh),\vec Q(\vlh))
  &\geq a_\edge(\vec Q(\vlh), \vec c\phi)^2
        a_\edge(\vec c\phi, \vec c\phi)^{-1}.
\end{align*}
Applying \cref{eq:Qeta2} yields
$a_\edge(\vec Q(\vlh),\vec c\,\phi) = -\vec c\cdot\vec F_\phi$
where
\begin{align*}\vec F_\phi
    \coloneqq
    \langle\phi,\vlh\nu\rangle_\edge
    - \int_\edge\bigl(\phi'\,\vec I + \phi\,\vix^\top\bigr)\vec Y(\vlh) \,\mathrm{d}x.
\end{align*}
Substituting this yields
\begin{align}\label{eq:dualbound}
  a_\edge(\vec Q(\vlh),\vec Q(\vlh))
    &\geq c_{\mathrm{q},\min} |\vec F_{\phi}|^2 \|\phi\|_\edge^{-2},
\end{align}
where we bound
\[|\vec F_\phi|^2 \geq \tfrac12 |\langle\phi,\vlh\nu\rangle_\edge|^2
-\Bigl|\int_\edge(\phi'\vec I + \phi\vix^\top)\vec Y(\vlh)\, \mathrm{d} x\Bigr|^2.\]
We choose
\begin{align*}
    \phi_0\equiv1,
    \qquad
    \phi_1(x)\coloneqq \tfrac{x}{h_\edge},
    \qquad
    \phi_2(x)\coloneqq 1-\tfrac{x}{h_\edge},
\end{align*}
compute $\|\phi_0\|_\edge^2=h_\edge$, and $\|\phi_i\|^2_\edge=\tfrac{h_\edge}3$,
$|\phi_i'|=h_\edge^{-1}$ for $i\geq 1$,
and apply the Cauchy--Schwarz inequality,
to bound the integral terms
\begin{align*}
    \Bigl|\int_\edge\phi_0\,\vix^\top\vec Y(\vlh)\,\mathrm{d}x\Bigr|^2
    &\leq h_\edge\norm{\vec Y(\vlh)}_\edge^2,
    \\
    \Bigl|\int_\edge(\phi_i'\vec I+\phi_i\vix^\top)\vec Y(\vlh)\,\mathrm{d}x\Bigr|^2
    &\leq h_\edge^{-1}\bigl(1+h_\edge\bigr)^{2}\norm{\vec Y(\vlh)}_\edge^2
    \leq C_\Omega h_\edge^{-1}\norm{\vec Y(\vlh)}_\edge^2
\end{align*}
for $i\geq 1$, with $C_\Omega\coloneqq(1+\diam\Omega)^2$;
here we used that $h_\edge\leq\diam\Omega$.

Let $c_0,c_{1}=c_{2}\in(0,1]$, and consider a convex combination of \cref{eq:dualbound}
for $\phi_0,\phi_1,\phi_2$ with weights $c_0/2$, $c_1/4$, $c_2/4$ summing to at most 1.
Substituting this into \cref{eq:lower_start} yields
\begin{align*}
    \hat A_\edge(\vlh,\vlh)
    &\geq
      \frac{c_0 c_{\mathrm q,\min}}{2}\mathcal{L}_{\graph,\edge}(\vlh,\vlh)
    + \frac{3c_1 c_{\mathrm q,\min}}{4h_\edge^{2}}\mathcal{M}_{\graph,\edge}(\vlh,\vlh)
      \\
    &\phantom{\geq} {}+ \Bigl(\frac{4}{(\Delta t)^2} c_{\mathrm y,\min}
            - \frac{c_0 c_{\mathrm q,\min}}{2}
            - \frac{3C_\Omega c_1 c_{\mathrm q,\min}}{2h_\edge^{2}}\Bigr)\norm{\vec Y(\vlh)}_\edge^{2}.
\end{align*}
The choices
\begin{align*}
    c_0\coloneqq\min\Bigl\{1,\frac{4c_{\mathrm y,\min}}{c_{\mathrm q,\min}(\Delta t)^{2}}\Bigr\},
    \qquad
    c_1\coloneqq\min\Bigl\{1,\frac{4c_{\mathrm y,\min}h_\edge^{2}}{3C_\Omega c_{\mathrm q,\min}(\Delta t)^{2}}\Bigr\}
\end{align*}
bound each of the two subtracted terms by $2c_{\mathrm y,\min}(\Delta t)^{-2}$ and hence render
the last bracket non-negative.
Using $\Delta t\leq T$ and $h_\edge^{-2}\gtrsim (\Delta t)^{-2}$,
the two remaining coefficients are bounded from below by
\begin{align*}
    \frac{c_0 c_{\mathrm q,\min}}{2}
    \geq \min\Bigl\{\frac{c_{\mathrm q,\min}}{2},\frac{2c_{\mathrm y,\min}}{T^{2}}\Bigr\},
    \qquad
    \frac{3c_1 c_{\mathrm q,\min}}{4h_\edge^{2}}
    \gtrsim \frac{\min\{3c_{\mathrm q,\min},\,4c_{\mathrm y,\min}C_\Omega^{-1}\}}{4(\Delta t)^{2}}.
\end{align*}
Summing over all edges $\edge\in\setEdge$
yields the desired lower bound of \cref{EQ:speceq}.
\end{proof}

We are now in a position to extend the classical condition number bounds for two-level overlapping Schwarz preconditioners~\cite{ToW05} to the network setting.
The corresponding proof depends on the aforementioned spectral equivalence result,
as well as on two constants:
the homogeneity constant $\sigma$,
measuring the variation of the network density across subdomains,
and the constant $\mu$ of a Friedrichs--Poincar\'e inequality on the network,
which quantifies its connectivity.
For their precise definitions, we refer to~\citep[Ass.~3.5, Lem.~3.6]{GoHeMa22},
while the bound on the condition number, stated below,
can be found in~\citep[Thm.~4.3]{GoHeMa22}.

\begin{theorem}[Condition number bound]\label{thm:condition}
Under the assumptions of \cref{thm:spectral_equi} and \citep[Ass.~3.5,\ Lem.~3.7]{GoHeMa22},
the condition number $\kappa$ of the preconditioned operator $P$ is bounded by
\begin{align*}
    \kappa \leq C_d \beta\alpha^{-1} \sigma\mu^2,
\end{align*}
where $C_d$ only depends on the dimension of the ambient space, and
$\alpha, \beta$ are the spectral equivalence constants of \cref{thm:spectral_equi}.
\end{theorem}

The proof is given in~\citep[Thm.~4.3]{GoHeMa22} for the static case, in which $\hat A$ is
spectrally equivalent to $\Lgraph$ alone, and carries over verbatim to the present
equivalence with $\Lgraph + (\Delta t)^{-2}\Mgraph$: the mass term is invariant
under the stable decomposition and is therefore bounded by
$\hat A$ on every subspace, while the coarse interpolation is estimated as before, so that
the constants remain independent of $\Delta t$.

\section{Numerical examples}\label{sec:numexp}

In this section, we present numerical experiments that support the theoretical predictions of this paper.
The code to reproduce the numerical experiments of this paper is available at \url{https://github.com/HyperHDG/}, and the data describing the fiber network model of paper used in the second set of numerical experiments can be found at~\cite{data2}.

\subsection{Optimal-order convergence of the method}

The first numerical example considers a toy problem to study the convergence properties of the proposed HDG method for Timoshenko beam networks. Specifically, we consider a network representing the three-dimensional coordinate cross
\[\big([-1,1]\times \{0\}\times \{0\}\big) \cup \big(\{0\}\times[-1,1]\times \{0\}\big) \cup \big(\{0\}\times \{0\}\times[-1,1]\big),\]
consisting of six arms of unit length aligned with the directions $\pm\vec e_1, \pm\vec e_2, \pm\vec e_3$. Before mesh refinement, the network consists of six edges and seven nodes. Dirichlet boundary conditions are imposed at the six nodes located at the tips of the cross. The bounding domain is chosen to be $\Omega = [-1,1]^3$.
To construct suitable data and corresponding solutions, we use the method of manufactured solutions. More precisely, we consider homogeneous material coefficients,
i.e., $\vec C_\mathrm{n}=\vec C_\mathrm{m}=\vec C_\mathrm{u}=\vec C_\mathrm{r}=\mathds{1}$, and choose the force terms as well as the Dirichlet and initial data such that the problem admits the following displacements and rotations as its solution:
\begin{align*}
	u_i(t,\vec x) &= A_i \sin(\omega \hat x)\cos(\omega t + a_i), &
	r_i(t,\vec x) &= B_i \sin(\omega \hat x)\cos(\omega t + b_i),
\end{align*}
for components $i=1,2,3$, where $\omega = 2\pi$ and $\hat x = x_1+x_2+x_3$ coincides on each arm with the signed axis coordinate, since exactly one coordinate is nonzero there. The amplitudes and phase shifts are chosen as
\begin{equation*}
	\vec A = (1,2,3), \quad \vec a = (0.3, 0.8, 1.3), \quad
	\vec B = (5,7,11), \quad \vec b = (0.5, 1.1, 1.7).
\end{equation*}
Since the profile $\sin(\omega\hat x)$ vanishes at the center vertex, displacements and rotations are continuous there, cf.~\eqref{EQ:timo_cont_hybrid}, and since the arms come in pairs of opposite orientation carrying a solution that is smooth across the center, the internal forces and moments satisfy the balance conditions \eqref{eq:balance}. The profile also vanishes at the six tips, so the Dirichlet data is identically zero for all times; in particular, the compatibility conditions hold to arbitrary order. The distinct per-component phase shifts ensure that all components of the primal and dual variables, as well as of their time derivatives, are nonzero, so that all couplings in \eqref{eq:elasticwave}, in particular the cross-product terms, are active in this numerical experiment.

We study the convergence behavior of the primal variables with respect to both discretization parameters,
using uniform refinement of the graph, i.e., repeated recursive subdivision of the edges, and uniform time steps.
Moreover, we repeat the experiments for different choices of polynomial degree $p$ and stabilization parameter $\tau$; see \cref{FIG:conv_plot}.

The observed convergence behavior agrees with the theoretical predictions of \cref{thm:fully_d}.
In time, cf.~\cref{FIG:conv_plot:dt}, the Crank--Nicolson scheme converges with order two,
as predicted, while the Gauss collocation methods with two and three stages attain their
classical orders four and six.
In space, cf.~\cref{FIG:conv_plot:h}, the primal variables converge with the predicted rate
$h^{p+1-|s|}=h^{p+1}$ for $\tau\simeq 1$ and all tested polynomial degrees $p=1,3,5$.
When the stabilization is varied at $p=3$, cf.~\cref{FIG:conv_plot:tau}, the predicted rate $h^{p}$
is attained for $\tau\simeq h$,
whereas for $\tau\simeq h^{-1}$ the observed rate exceeds the predicted one by at least one order,
as discussed in \cref{rem:sharpness}.
The numerical experiments further indicate that the hybrid variable converges with order $p+2$
in the norm induced by $\Mgraph$.
This resembles the superconvergence of numerical traces established for related HDG methods
in~\cite{Cockburn2008}.

\begin{figure}
  \newcommand{\figurewidth}{4.6cm}\newcommand{\figureheight}{4.6cm}\centering
  \hspace*{46pt}\begin{subfigure}[b]{\figurewidth}
    \providecommand{\figurewidth}{\linewidth}
\providecommand{\figureheight}{\linewidth}
\providecommand{\plotdatadir}{}
\providecolor{mplC0}{HTML}{1F77B4}\providecolor{mplC1}{HTML}{FF7F0E}\providecolor{mplC2}{HTML}{2CA02C}\providecolor{mplC3}{HTML}{D62728}\providecolor{mplC4}{HTML}{9467BD}\providecolor{mplC5}{HTML}{8C564B}\begin{tikzpicture}[baseline, trim axis left, trim axis right]
\begin{axis}[
  scale only axis,
  width=\figurewidth,
  height=\figureheight,
  cycle list={{mplC0},{mplC1},{mplC2},{mplC3},{mplC4},{mplC5}},
  xmode=log, log basis x={2},
  ymode=log, log basis y={10},
  ymin=1e-11, ymax=10,
  xlabel={time step $\Delta t$\strut},
  ylabel={max rel. $L^2$ error},
  legend pos=south east,
  legend style={legend cell align=left},
  legend entries={
    {\hspace{-.6cm}stages},
    {1},
    {2},
    {3}
  }
]
\addlegendimage{empty legend}
\foreach \i in {0,...,2}{
  \addplot+[
    mark=+,
    unbounded coords=discard,
    x filter/.expression={\thisrow{series} == \i ? \pgfmathresult : nan},
  ] table[x=nt, y=e_rel, col sep=comma] {\plotdatadir ne9-11-conv-t-stages-hom.csv};
}
\addplot[gray, thin, mark=none, forget plot] coordinates {(0.03125,0.004) (0.0625,0.004) (0.0625,0.016) (0.03125,0.004)};
\node[below, gray, font=\footnotesize] at (axis cs:0.04419417382415922,0.004) {1};
\node[right, gray, font=\footnotesize] at (axis cs:0.0625,0.008) {2};
\addplot[gray, thin, mark=none, forget plot] coordinates {(0.03125,1.5e-06) (0.0625,1.5e-06) (0.0625,2.4e-05) (0.03125,1.5e-06)};
\node[below, gray, font=\footnotesize] at (axis cs:0.04419417382415922,1.5e-06) {1};
\node[right, gray, font=\footnotesize] at (axis cs:0.0625,6e-06) {4};
\addplot[gray, thin, mark=none, forget plot] coordinates {(0.03125,4e-10) (0.0625,4e-10) (0.0625,2.56e-08) (0.03125,4e-10)};
\node[below, gray, font=\footnotesize] at (axis cs:0.04419417382415922,4e-10) {1};
\node[right, gray, font=\footnotesize] at (axis cs:0.0625,3.2e-09) {6};
\end{axis}
\end{tikzpicture}\caption{varying number of stages}\label{FIG:conv_plot:dt}
  \end{subfigure}\hspace{20pt}\begin{subfigure}[b]{\figurewidth}
    \providecommand{\figurewidth}{\linewidth}
\providecommand{\figureheight}{\linewidth}
\providecommand{\plotdatadir}{}
\providecolor{mplC0}{HTML}{1F77B4}\providecolor{mplC1}{HTML}{FF7F0E}\providecolor{mplC2}{HTML}{2CA02C}\providecolor{mplC3}{HTML}{D62728}\providecolor{mplC4}{HTML}{9467BD}\providecolor{mplC5}{HTML}{8C564B}\begin{tikzpicture}[baseline, trim axis left, trim axis right]
\begin{axis}[
  scale only axis,
  width=\figurewidth,
  height=\figureheight,
  cycle list={{mplC0},{mplC1},{mplC2},{mplC3},{mplC4},{mplC5}},
  xmode=log, log basis x={2},
  ymode=log, log basis y={10},
  ymin=1e-11, ymax=10,
  xlabel={discretization size $h$\strut},
  yticklabels={},
  legend pos=south east,
  legend style={legend cell align=left},
  legend entries={
    {\hspace{-.6cm}$\tau\sim$},
    {$h^{-1}$},
    {$1$},
    {$h$}
  }
]
\addlegendimage{empty legend}
\foreach \i in {0,...,2}{
  \addplot+[
    mark=+,
    unbounded coords=discard,
    x filter/.expression={\thisrow{series} == \i ? \pgfmathresult : nan},
  ] table[x=nx, y=e_rel, col sep=comma] {\plotdatadir ne9-03-conv-tau.csv};
}
\addplot[gray, thin, mark=none, forget plot] coordinates {(0.03125,0.0002) (0.015625,0.0002) (0.015625,2.5e-05) (0.03125,0.0002)};
\node[above, gray, font=\footnotesize] at (axis cs:0.02209708691207961,0.0002) {1};
\node[left, gray, font=\footnotesize] at (axis cs:0.015625,7.071067811865475e-05) {3};
\addplot[gray, thin, mark=none, forget plot] coordinates {(0.015625,1e-09) (0.03125,1e-09) (0.03125,1.6e-08) (0.015625,1e-09)};
\node[below, gray, font=\footnotesize] at (axis cs:0.02209708691207961,1e-09) {1};
\node[right, gray, font=\footnotesize] at (axis cs:0.03125,4e-09) {4};
\end{axis}
\end{tikzpicture}\caption{varying stabilization $\tau$}\label{FIG:conv_plot:tau}
  \end{subfigure}\hspace*{8pt}
  \par\vspace{.8cm}
  \hspace*{46pt}\begin{subfigure}[b]{\figurewidth}
    \providecommand{\figurewidth}{\linewidth}
\providecommand{\figureheight}{\linewidth}
\providecommand{\plotdatadir}{}
\providecolor{mplC0}{HTML}{1F77B4}\providecolor{mplC1}{HTML}{FF7F0E}\providecolor{mplC2}{HTML}{2CA02C}\providecolor{mplC3}{HTML}{D62728}\providecolor{mplC4}{HTML}{9467BD}\providecolor{mplC5}{HTML}{8C564B}\begin{tikzpicture}[baseline, trim axis left, trim axis right]
\begin{axis}[
  scale only axis,
  width=\figurewidth,
  height=\figureheight,
  cycle list={{mplC0},{mplC1},{mplC2},{mplC3},{mplC4},{mplC5}},
  xmode=log, log basis x={2},
  ymode=log, log basis y={10},
  ymin=1e-12, ymax=10,
  xlabel={discretization size $h$\strut},
  ylabel={max rel. $L^2$ error},
  legend pos=south east,
  legend style={legend cell align=left},
  legend entries={
    {\hspace{-.6cm}deg},
    {1},
    {3},
    {5}
  }
]
\addlegendimage{empty legend}
\foreach \i in {0,...,2}{
  \addplot+[
    mark=+,
    unbounded coords=discard,
    x filter/.expression={\thisrow{series} == \i ? \pgfmathresult : nan},
  ] table[x=nx, y=e_rel, col sep=comma] {\plotdatadir ne9-12-conv-x-hom.csv};
}
\addplot[gray, thin, mark=none, forget plot] coordinates {(0.03125,0.005) (0.0625,0.005) (0.0625,0.02) (0.03125,0.005)};
\node[below, gray, font=\footnotesize] at (axis cs:0.04419417382415922,0.005) {1};
\node[right, gray, font=\footnotesize] at (axis cs:0.0625,0.01) {2};
\addplot[gray, thin, mark=none, forget plot] coordinates {(0.03125,8e-07) (0.0625,8e-07) (0.0625,1.28e-05) (0.03125,8e-07)};
\node[below, gray, font=\footnotesize] at (axis cs:0.04419417382415922,8e-07) {1};
\node[right, gray, font=\footnotesize] at (axis cs:0.0625,3.2e-06) {4};
\addplot[gray, thin, mark=none, forget plot] coordinates {(0.03125,6e-11) (0.0625,6e-11) (0.0625,3.84e-09) (0.03125,6e-11)};
\node[below, gray, font=\footnotesize] at (axis cs:0.04419417382415922,6e-11) {1};
\node[right, gray, font=\footnotesize] at (axis cs:0.0625,4.8e-10) {6};
\end{axis}
\end{tikzpicture}\caption{varying degree $p$}\label{FIG:conv_plot:h}
  \end{subfigure}\hspace{20pt}\begin{subfigure}[b]{\figurewidth}
    \providecommand{\figurewidth}{\linewidth}
\providecommand{\figureheight}{\linewidth}
\providecommand{\plotdatadir}{}
\providecolor{mplC0}{HTML}{1F77B4}\providecolor{mplC1}{HTML}{FF7F0E}\providecolor{mplC2}{HTML}{2CA02C}\providecolor{mplC3}{HTML}{D62728}\providecolor{mplC4}{HTML}{9467BD}\providecolor{mplC5}{HTML}{8C564B}\begin{tikzpicture}[baseline, trim axis left, trim axis right]
\begin{axis}[
  scale only axis,
  width=\figurewidth,
  height=\figureheight,
  cycle list={{mplC0},{mplC1},{mplC2},{mplC3},{mplC4},{mplC5}},
  xmode=log, log basis x={2},
  ymode=log, log basis y={10},
  ymin=1e-12, ymax=10,
  xlabel={discretization size $h$\strut},
  yticklabels={},
  legend pos=south west,
  legend style={legend cell align=left},
  legend entries={
    {\hspace{-.6cm}deg},
    {1},
    {3},
    {5}
  }
]
\addlegendimage{empty legend}
\foreach \i in {0,...,2}{
  \addplot+[
    mark=+,
    unbounded coords=discard,
    x filter/.expression={\thisrow{series} == \i ? \pgfmathresult : nan},
  ] table[x=nx, y=e_trace, col sep=comma] {\plotdatadir ne9-12-conv-trace-hom.csv};
}
\addplot[gray, thin, mark=none, forget plot] coordinates {(0.015625,0.00015) (0.03125,0.00015) (0.03125,0.0012) (0.015625,0.00015)};
\node[below, gray, font=\footnotesize] at (axis cs:0.02209708691207961,0.00015) {1};
\node[right, gray, font=\footnotesize] at (axis cs:0.03125,0.0004242640687119285) {3};
\addplot[gray, thin, mark=none, forget plot] coordinates {(0.03125,1e-09) (0.0625,1e-09) (0.0625,3.2e-08) (0.03125,1e-09)};
\node[below, gray, font=\footnotesize] at (axis cs:0.04419417382415922,1e-09) {1};
\node[right, gray, font=\footnotesize] at (axis cs:0.0625,5.6568542494923804e-09) {5};
\addplot[gray, thin, mark=none, forget plot] coordinates {(0.0625,1.5e-11) (0.125,1.5e-11) (0.125,1.92e-09) (0.0625,1.5e-11)};
\node[below, gray, font=\footnotesize] at (axis cs:0.08838834764831845,1.5e-11) {1};
\node[right, gray, font=\footnotesize] at (axis cs:0.125,1.697056274847714e-10) {7};
\end{axis}
\end{tikzpicture}\caption{error in hybrid variable $\vec \lambda$}\label{FIG:conv_plot:trace}
  \end{subfigure}\hspace*{8pt}
  \caption{Maximum $L^2$-errors for the primal variables in
    \subref{FIG:conv_plot:dt}--\subref{FIG:conv_plot:h},
    and for the hybrid variable $\vlh$, measured in the norm induced by $\Mgraph$,
    in \subref{FIG:conv_plot:trace}.
  Panel \subref{FIG:conv_plot:dt} shows the error with respect to
  the time step size $\Delta t$, varying the number of stages of the Gauss collocation method.
  The remaining panels show the error with respect to the discretization size $h$ (beam length):
  panel \subref{FIG:conv_plot:tau} varies the stabilization parameter $\tau\simeq h^s$
  with $s=-1,0,1$ (blue, orange, green) at $p=3$,
  while panels \subref{FIG:conv_plot:h} and \subref{FIG:conv_plot:trace}
  vary the polynomial degree $p=1,3,5$ (blue, orange, green) at $\tau=1$.
  }
   \label{FIG:conv_plot}
\end{figure}

\subsection{Verification of the spectral properties of the preconditioner}

The purpose of the second numerical experiment is to verify the theoretical
properties of the preconditioner introduced in \cref{sec:precond}, namely the
robustness of the preconditioned conjugate gradient method with respect to the
subdomain size and to the time step size, cf.~\cref{thm:condition,thm:spectral_equi}.
As a realistic test case,
we simulate elastic wave propagation in an approximately
$8\,\mathrm{mm} \times 8\,\mathrm{mm}$ piece of paper, for which our
collaborators at the Fraunhofer-Chalmers Centre (FCC) provided the spatial
network and material parameters.

The spatial network consists of about 1.72M nodes and 3.31M edges. The initial data is the network at rest, while the right-hand side $\vf_\mathrm{n}$ is constant. All outer boundaries are clamped with trace displacement and rotation set to zero.
We use polynomial degree 3 on all edges for the HDG discretization.
The resulting linear system of equations is solved by the preconditioned conjugate gradient method
with the domain decomposition preconditioner as described in \cref{sec:precond}.
The number of subdomains is varied in the experiments below, and the corresponding local problems are solved independently and fully in parallel, each using a direct solver.

Two frames of the simulation are shown in \cref{FIG:paper_wave:frames}.
\Cref{FIG:paper_wave:iters} shows the convergence of the preconditioned conjugate gradient method for varying relative subdomain diameter $H$:
the iteration counts remain nearly uniform, as predicted by \cref{thm:condition}.
The mild increase of iterations required to reach the same relative error in the energy norm
is due to a deterioration of the homogeneity constant $\sigma$ and of the constant $\mu$ related to the Friedrichs--Poincar\'e inequality; see \cref{tab:constants} for estimates computed following~\citep[Sec.~6.2]{GoHeMa22}.
In \cref{FIG:paper_wave:dt}, the iteration counts are seen to remain robust as the time step is reduced, in accordance with the $\Delta t$-independence of the constants in \cref{thm:spectral_equi,thm:condition}. The observed improvement for smaller time steps is not predicted by these bounds; it reflects the increasing dominance of the mass term $(\Delta t)^{-2}\Mgraph$, which renders the condensed operator progressively more local.

\begin{figure}
  \begin{minipage}[c]{.55\linewidth}
    \centering
    \includegraphics[width=\linewidth]{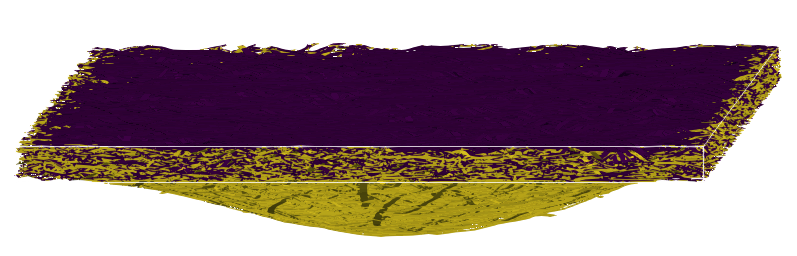}
    \caption{Two frames of a zoomed-in simulation
      at times 0 and $T/2$ are shown, where $T=2.3\,\mu\mathrm{s}$ is the period.}
    \label{FIG:paper_wave:frames}
  \end{minipage}\hfill
  \begin{minipage}[c]{.45\linewidth}
    \centering
    \captionof{table}{Estimates of constants in \cref{thm:condition}.}
    \label{tab:constants}
    \begin{filecontents*}{gortz-fiber2raw.csv}
# extent_x = 8000.4
# extent_y = 8000.78
# R0 = 62.506
# lc_min_p2 = 1420.1
# lc_min_p50 = 1713.53
# lc_min_p98 = 6203.75
# c_long_p50 = 5.16398e+09
# c_shear_p50 = 1.7224e+09
# t_cross = 4.64513e-06
# T1_shear = 6.56921e-06
# T1_bend = 6.90383e-06
# T1_est = 9.52981e-06
n,R,Rinv,sigma,sigma_sqrt,cell_min,cell_max,cells_empty,cells_total,R_over_lc_p50,R_over_lc_p50_sq,R_over_lc_p2,R_over_lc_p2_sq,mu,mu_mean,avg_invlam2,mu_eval,mu_viol,mu_skip
4,1000.0980014274764,0.0009999020081758621,1.1073737023777641,1.052318251470421,2395490.5194835546,2652703.2055713376,0,16,0.5836487943230171,0.3406459151147115,0.7042451370426722,0.4959612130482522,0.6294486316041112,0.6038519244265088,1459476.543580411,16,0,0
8,500.0490007137382,0.0019998040163517242,1.2519744376563393,1.118916635704528,568610.313696344,711885.577735575,0,64,0.29182439716150854,0.08516147877867787,0.3521225685213361,0.12399030326206305,0.7021681022702488,0.6409245278260329,411336.596881167,64,1,0
16,250.0245003568691,0.0039996080327034484,1.5810166753653128,1.2573848557085905,129706.86836163282,205068.72178915498,0,256,0.14591219858075427,0.021290369694669468,0.17606128426066806,0.030997575815515762,0.8340142887428281,0.7166058528364866,128689.66241347406,256,3,0
32,125.01225017843456,0.007999216065406897,2.370502795602126,1.5396437235939118,27837.53783556022,65988.96126187546,0,1024,0.07295609929037714,0.005322592423667367,0.08803064213033403,0.0077493939538789405,1.2445768276544888,0.9816070510063046,60503.267892858225,1024,5,0
\end{filecontents*}
\pgfplotstabletypeset[
      col sep=comma,
      comment chars={\#},
      columns={n,sigma,mu},
      columns/n/.style={column name={$H^{-1}$}, int detect},
      columns/R/.style={column name={$R$}, fixed, fixed zerofill, precision=1, 1000 sep={}},
      columns/sigma/.style={column name={$\sigma$}, fixed, fixed zerofill, precision=2},
      columns/mu/.style={column name={$\mu$}, fixed, fixed zerofill, precision=2},
      every head row/.style={before row=\toprule, after row=\midrule},
      every last row/.style={after row=\bottomrule},
    ]{gortz-fiber2raw.csv}
  \end{minipage}
\end{figure}

\begin{figure}
  \centering
  \hspace*{50pt}\begin{subfigure}{4.8cm}
  \providecommand{\figurewidth}{\linewidth}
\providecommand{\plotdatadir}{}
\begin{tikzpicture}[trim axis left, trim axis right]
\begin{axis}[
  scale only axis, width=4.8cm, height=4.2cm,
  ymode=log, log basis y={10},
  ymin=1e-10, ymax=10,
  xlabel={iteration},
  ylabel={$\|u - u^{(\ell)}\|_{\text{E,rel}}$},
  legend pos=north east,
  legend style={font=\footnotesize, legend cell align=left},
  legend entries={
    {\hspace{-.6cm}H},
    {1/32},
    {1/16},
    {1/8},
    {1/4},
  }
]
\addlegendimage{empty legend}
\foreach \i in {1,0,3,2}{
  \addplot+[
    mark=,
    unbounded coords=discard,
    x filter/.expression={\thisrow{series} == \i ? \pgfmathresult : nan},
  ] table[x=it, y=enorm, col sep=comma] {\plotdatadir ne18-26-fiber-timowave-H.csv};
}
\end{axis}
\end{tikzpicture}
  \subcaption{\label{FIG:paper_wave:iters} varying subdomain size}
  \end{subfigure}\hfill
  \begin{subfigure}{4.8cm}
  \providecommand{\figurewidth}{\linewidth}
\providecommand{\plotdatadir}{}
\begin{tikzpicture}[trim axis left, trim axis right]
\begin{axis}[
  scale only axis, width=4.8cm, height=4.2cm,
  ymode=log, log basis y={10},
  ymin=1e-10, ymax=10,
  xlabel={iteration},
  yticklabels={},
  legend pos=north east,
  legend style={font=\footnotesize},
  cycle list={blue, red, brown!60!black, black, green!60!black},
  legend entries={
    {T/2},
    {T/4},
    {T/8},
    {T/16},
    {T/32}
  }
]
\foreach \i in {1,3,4,0,2}{
  \addplot+[
    mark=,
    unbounded coords=discard,
    x filter/.expression={\thisrow{series} == \i ? \pgfmathresult : nan},
  ] table[x=it, y=enorm, col sep=comma] {\plotdatadir ne18-25-fiber-timowave-dt.csv};
}
\end{axis}
\end{tikzpicture}
  \subcaption{\label{FIG:paper_wave:dt} varying time step size}
  \end{subfigure}\hspace*{8pt}
  \caption{
    The relative error in the energy norm
    with respect to the iteration number is shown.
    \cref{FIG:paper_wave:iters} shows results for varying relative subdomain diameter,
    for a single time step equal to half a period $T/2=1.15\,\mu\mathrm{s}$,
    while \cref{FIG:paper_wave:dt} shows varying time step sizes relative to the period for
    $H=1/16$.
  }
\end{figure}

\subsection{Realistic experiment}

The purpose of the final experiment is to demonstrate that the proposed
solver scales to industrially relevant problem sizes.
We consider a $20\,\mathrm{mm}\times20\,\mathrm{mm}$
fiber network realization comprising about 10.2M nodes and 14.6M edges,
again provided by the FCC.
All nodes within $0.4\,\mathrm{mm}$ of
the four lateral boundaries are clamped, with trace displacement and rotation
set to zero. Starting from the network at rest, the sheet is excited by a tap
at its center: a transverse body force, Gaussian in space and time with
widths $\sigma_x = 250\,\mu\mathrm{m}$ and $\sigma_t = 0.15\,\mu\mathrm{s}$.

We again use polynomial degree 3 for the HDG discretization, resulting in a
trace system of $6.1\cdot10^7$ unknowns with $1.3\cdot10^9$ nonzero entries,
and the Crank--Nicolson scheme with time step
$\Delta t = \sigma_t/2 = 75\,\mathrm{ns}$, simulating until
$T = 6.3\,\mu\mathrm{s}$, shortly before the wave front reaches the clamped
boundary. The preconditioner uses an artificial coarse mesh of $16\times16$
cells, inducing 256 subdomains of diameter $H\approx1.2\,\mathrm{mm}$.
The subdomain problems are solved by sparse Cholesky
factorizations, and the coarse problem by a direct solve.

Across all 84 time steps, the preconditioned conjugate gradient method requires on average 44.0 iterations to reach a relative tolerance of $10^{-9}$; one iteration takes about 4.2 seconds on 128 MPI ranks of a
dual-socket AMD EPYC~7713 system. The Krylov iterations thus account for 90\% of the total runtime of about five hours, with the preconditioner
setup taking 141 seconds and the remaining 21 seconds per time step spent on local reconstruction, error and energy evaluation, and the in-situ output of the trajectory and per-edge energies (61\,GB in total). \Cref{FIG:large_exp:final} shows the transverse
displacement at the final time: the ellipsoid wave front has traveled about $9.6\,\mathrm{mm}$ from the tap, reflecting the anisotropy of the fiber orientation. The required memory is around 840\,GB, with the direct local solves for the subdomains accounting for over 90\%.
\Cref{FIG:large_exp:energy} shows the
evolution of the discrete energy \eqref{eq:disc_energy}, split into strain,
kinetic, and hybrid contributions. The tap injects energy until about
$1\,\mu\mathrm{s}$; afterwards strain and kinetic energy equilibrate, the
hybrid (stabilization) part remains below $0.2\%$ of the total, and the
total energy is essentially conserved, decaying by about $1.3\%$ over the
remaining $5\,\mu\mathrm{s}$. This drift does not contradict the exact
conservation established in \cref{sec:energy_cons}, which presumes exact
solution of the linear system \eqref{eq:saddle_fully_discrete} at each step:
the algebraic error admitted by the Krylov tolerance of $10^{-9}$, amplified
by the conditioning of the multiscale trace system, perturbs the conservative
scheme slightly at every step and accumulates over the 84 steps.

\begin{figure}
  \centering
  \begin{subfigure}[b]{.48\linewidth}
    \centering
    \includegraphics[width=.92\linewidth]{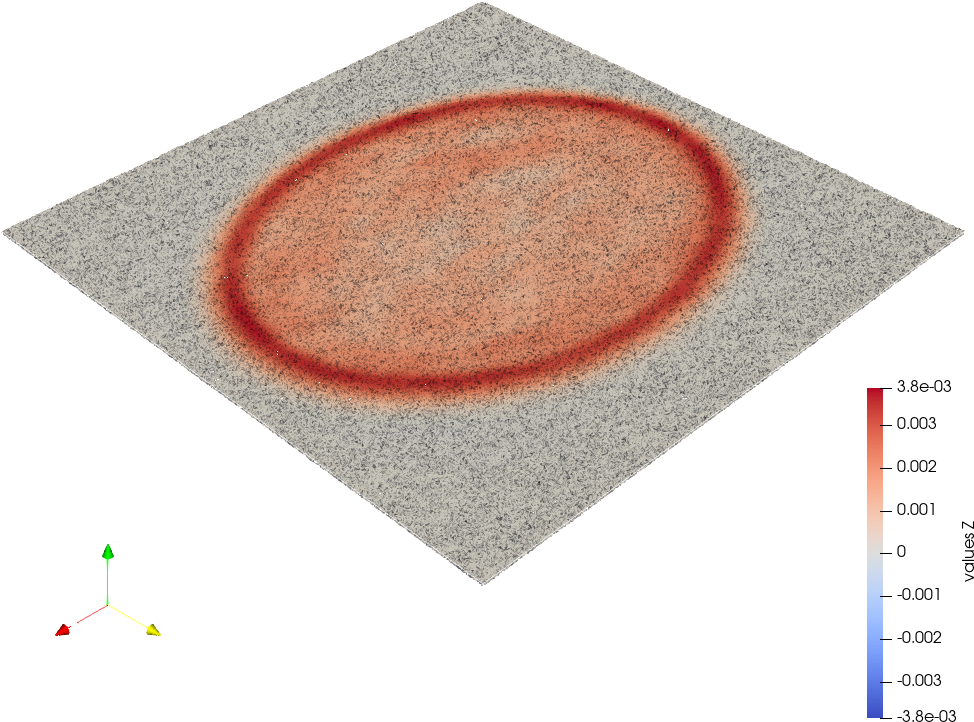}
    \caption{displacement at final time $T$}\label{FIG:large_exp:final}
  \end{subfigure}\hspace*{10pt}
  \begin{subfigure}[b]{.45\linewidth}
    \centering
    {\renewcommand{\plotdatadir}{}\newcommand{\figurewidth}{.82\linewidth}\providecommand{\figurewidth}{\linewidth}
\providecommand{\plotdatadir}{}
\begin{tikzpicture}
\begin{axis}[
  width=\figurewidth,
  xlabel={$t$ [s]},
  ylabel={energy [\%]},
  legend pos=outer north east,
  legend columns=1,
  legend style={font=\footnotesize},
]
\addplot+[mark=none, thick] table[x=t, y expr=100*\thisrow{strain}/108.9069427, col sep=comma] {\plotdatadir ne18-33-net3-clamp2-energy4.csv};
\addlegendentry{strain}
\addplot+[mark=none, thick] table[x=t, y expr=100*\thisrow{kinetic}/108.9069427, col sep=comma] {\plotdatadir ne18-33-net3-clamp2-energy4.csv};
\addlegendentry{kinetic}
\addplot+[mark=none, thick] table[x=t, y expr=100*\thisrow{hybrid}/108.9069427, col sep=comma] {\plotdatadir ne18-33-net3-clamp2-energy4.csv};
\addlegendentry{hybrid}
\addplot[black, thick] table[x=t, y expr=100*\thisrow{total}/108.9069427, col sep=comma] {\plotdatadir ne18-33-net3-clamp2-energy4.csv};
\addlegendentry{total}
\end{axis}
\end{tikzpicture}}\caption{energy over time}\label{FIG:large_exp:energy}
  \end{subfigure}
  \caption{Wave propagation on the large network.
    \cref{FIG:large_exp:final} shows the network warped by the displacement
    field (scaled by $4.6\cdot10^4$) and colored by the transverse component
    at the final time $T=6.3\,\mu\mathrm{s}$.
    \cref{FIG:large_exp:energy} shows the strain, kinetic, hybrid
    (stabilization), and total energy over time, as percentages of the maximum
    total energy, which is attained at the end of the tap.}
  \label{FIG:large_exp}
\end{figure}

\section{Conclusions and future work}\label{sec:conclusion}

We proposed and analyzed a hybridizable discontinuous Galerkin method for elastic
wave propagation in Timoshenko beam networks, a model for the dynamics of fiber-based materials such as paper.
The resulting semi-discrete system was discretized in time using a Crank--Nicolson scheme.
Exploiting the zero-dimensional nature of the network junctions, the global system at each time step was condensed to a system posed only on the network nodes, with a size independent of the polynomial degree. An implicit scheme was chosen because the large disparity of fiber segment lengths renders the CFL condition of explicit schemes prohibitive.

We carried out a projection-based a priori error analysis, establishing optimal-order convergence of the semi-discrete scheme and of the fully discrete scheme. The formulation is energy-conservative, and
the Crank--Nicolson discretization preserves this property exactly.
For the typically ill-conditioned condensed systems, we introduced a two-level overlapping additive Schwarz preconditioner on an artificial coarse mesh, whose analysis is based on a spectral equivalence between the condensed HDG operator and weighted graph operators. The numerical experiments confirmed the optimal convergence orders and the robustness of the preconditioner with respect to subdomain size and time step size, as well as scalability to industrially relevant problem sizes.

Several directions are left for future work. On the algorithmic side, we plan to investigate multilevel variants of the preconditioner based on algebraic coarsening strategies and requiring less overlap between subdomains. On the modeling side, the linearly elastic beam network could be extended to account for nonlinear phenomena, such as geometric nonlinearities, or nonlinear constitutive laws.

\section*{Acknowledgments}
The authors would like to thank M. Görtz (Fraunhofer-Chalmers Centre) for valuable discussions on the construction of the fiber networks, for providing network realizations, and for their support with troubleshooting.
The authors would like to thank R. Schnaubelt (Karlsruhe Institute of Technology) as well, for valuable discussions on regularity and well-posedness.
M.~Hauck and J.~Holten acknowledge funding from the Deut\-sche Forschungsgemeinschaft (DFG, German Research Foundation) -- Project-ID 258734477 -- SFB 1173. 
Furthermore,  A.~M\aa lqvist and L.~Swoboda acknowledge funding from the Swedish Research Council (VR) -- Project-ID 2023-03258\_VR.
A.\ Rupp has been supported by the Deutsche Forschungsgemeinschaft (DFG, German
Research Foundation) -- 577175348.
Parts of this work were conducted during the authors’ stay at the
Hausdorff Research Institute for Mathematics funded by the Deutsche Forschungsgemeinschaft (DFG, German Research Foundation) under Germany's Excellence Strategy – EXC-2047/2 – 390685813.

The authors used ChatGPT and Claude Code with Fable 5 and Opus 5 to assist in refining the language of the manuscript and in refining and checking arguments in the proofs. Claude Code with Fable 5 and Opus 5 was additionally used to assist in the implementation of the numerical experiments.

\appendix

\section{Well-posedness}\label{sec:wellposed}

\begin{proof}[Proof of \cref{lem:hdm_well}]
\resetproofsteps
We outline the standard Rothe argument.

\proofstep{Discretization in time}\label{stp:wp_disc}
We discretize \cref{eq:saddle} in time by the Crank--Nicolson scheme of \cref{sec:fullydisc}, applied without discretizing in space and with the notation introduced there.
Since the dual and hybrid variables carry no initial data of their own, they are placed at the staggered points as \((\vec q^{k+\frac12},\vlh^{k+\frac12})_k\), while \((\vec y^k)_k\) sits at the grid points \(t_k=k\Delta t\); the scheme is initialized from \(\vec y_0\), \(\vec y_1\) and \eqref{eq:comp_cond}.

\proofstep{Solvability}\label{stp:wp_solv}
Each time step is uniquely solvable by the condensation argument of \cref{lem:fd_well}: the equations decouple into edge-local problems, which are the static problem of \citep[Lem.~3.2]{hauck2025} perturbed by the coercive term \((\Delta t)^{-2}c(\cdot,\cdot)\), and the condensed nodal problem is symmetric positive definite.

\proofstep{Energy identity}\label{stp:wp_energy}
The discrete energy \[\mathfrak E^{k+\frac12}\coloneqq\tfrac12 a(\vec q^{k+\frac12},\vec q^{k+\frac12})+\tfrac12 c(D_t^{k+\frac12}\vec y,D_t^{k+\frac12}\vec y)\] obeys the analogue of the identity of \cref{lem:disc_energy}, driven only by the source term and the Dirichlet data.

\proofstep{Nodal trace estimate}\label{stp:wp_trace}
Though not controlled by \(\mathfrak E^{k+\frac12}\) at any fixed time, the nodal traces of the dual variable satisfy
\[
\Delta t\sum_{k}\|\vec q^{k}\|_{\setEdge}^2
\lesssim
(1+T)\Emax+\|\vec f\|_{L^2(0,T;\Ltwo)}^2,
\qquad
\Emax\coloneqq\max_k\mathfrak E^{k+\frac12},
\]
the time-discrete counterpart of the hidden-regularity estimates for hyperbolic problems;
we follow the multiplier argument of~\citep[Thm.~2.1]{LasieckaLionsTriggiani86},
see also~\cite{Lions88,MillaMirandaMedeiros88}, with an edgewise affine multiplier.

\proofstep{Energy bound}\label{stp:wp_bound}
Summing the identity of \cref{stp:wp_energy} and bounding the source term by a Riemann sum and the Dirichlet term by the Cauchy--Schwarz inequality and \cref{stp:wp_trace}, Young's inequality bounds \(\Emax\) by the data, uniformly in \(\Delta t\).

\proofstep{Passage to the limit}\label{stp:wp_limit}
By linearity, the difference quotients satisfy the same scheme with differenced data, so that \crefrange{stp:wp_energy}{stp:wp_bound} bound them uniformly in \(\Delta t\); the momentum and constitutive equations bound the remaining norms of the dual and hybrid variable.
A standard weak-\(*\) compactness argument for the interpolants yields a limit solving \cref{eq:saddle} with the asserted regularity and initial values; uniqueness follows as in \cref{lem:fd_well}, since vanishing data force the energy of \cref{stp:wp_energy}, and hence the solution, to vanish.
\end{proof}

\bibliographystyle{alpha}
\bibliography{bib.bib}

\typeout{get arXiv to do 4 passes: Label(s) may have changed. Rerun}
\end{document}